\documentclass[12pt]{amsart}

\usepackage{mathrsfs}
\usepackage{amsthm}
\usepackage{enumitem}
\usepackage{hyperref}
\usepackage{amssymb}

\usepackage{amsmath, amssymb, amsfonts, amsbsy}
\usepackage[margin=2.5cm]{geometry}
\usepackage[dvipsnames]{xcolor}

\newif\ifshowrev \showrevtrue

\makeatletter

\theoremstyle{plain}
\newtheorem{thm}{\protect\theoremname}[section]
  \theoremstyle{plain}
  \newtheorem{lem}[thm]{\protect\lemmaname}
  \theoremstyle{plain}
\newtheorem{prop}[thm]{\protect\propositionname}  \theoremstyle{plain}
  \newtheorem{fact}[thm]{\protect\factname}
  \theoremstyle{definition}
  \newtheorem{defn}[thm]{\protect\definitionname}
  \theoremstyle{remark}
  \newtheorem*{acknowledgement*}{\protect\acknowledgementname}
  \theoremstyle{remark}
  \newtheorem*{claim*}{\protect\claimname}
  \theoremstyle{remark}
  \newtheorem{rem}[thm]{\protect\remarkname}
  \theoremstyle{plain}
  \newtheorem{cor}[thm]{\protect\corollaryname}

 \newtheorem{thmx}{Theorem}
 
\newtheorem{lemx}[thmx]{Lemma}
\newtheorem{propx}[thmx]{Proposition}

\global\long\def\ran#1{\operatorname{ran}\left(#1\right)}
\newtheorem{claimnum}[thm]{\protect\claimname}

\usepackage{babel}

\makeatother

  \providecommand{\acknowledgementname}{Acknowledgement}
  \providecommand{\claimname}{Claim}
  \providecommand{\corollaryname}{Corollary}
  \providecommand{\definitionname}{Definition}
  \providecommand{\factname}{Fact}
  \providecommand{\lemmaname}{Lemma}
  \providecommand{\remarkname}{Remark}
   \newtheorem{question}[thm]{\protect\questionname}
    \providecommand{\propositionname}{Proposition}
\providecommand{\theoremname}{Theorem}
  \providecommand{\questionname}{Question}

\begin{document}
 \global\long\def\dim#1{\mathrm{d}\left(#1\right)}
 \global\long\def\cl#1{\mathrm{cl}\left(#1\right)}
 \global\long\def\gcl#1{\mathrm{gcl}\left(#1\right)}
 
 \global\long\def\acl#1{\mathrm{acl}\left(#1\right)}

 \global\long\def\acleq#1{\mathrm{acl}^{eq}\left(#1\right)}

  \global\long\def\Th {\mathrm{Th}}
 \global\long\def\dom#1{\operatorname{dom}\left(#1\right)}
 \global\long\def\Aut{\mathrm{Aut}}
 \global\long\def\Sym{\mathrm{Sym}}
 \global\long\def\Autf{\mathrm{Autf}}

  \global\long\def\PK#1#2{\left({#1\atop #2}\right)}
 
\begin{abstract}
Let $M$ be a countable ab-initio Hrushovski generic structure obtained from
the pre-dimension function $\delta(A)=\mathfrak m|A|-|R(A)|$, with $\mathfrak m\geq 2$ an integer. 
We show that $\Aut(M)$, with the point-wise convergence topology, has the small index property: every subgroup of index strictly less than $2^{\aleph_0}$ contains the point-wise stabiliser of a finite subset of $M$. 
The class of finite substructures of $M$ is known not to have the extension property for partial automorphisms, so the usual route to the small index property is unavailable. 
Instead we adapt a result of Lascar  on strongly minimal sets to show that every subgroup of small index contains the
point-wise stabiliser of the geometric closure of a finite set. We combine this with the extension property for the subclass of structures of pre-dimension zero, which we establish here.

\end{abstract}

\title[SIP of AUT-Groups of ab-initio generics]{The small index property of automorphism groups of ab-initio generic
structures}

\author{Zaniar Ghadernezhad}

\address{School of Physical and Chemical Sciences, Queen Mary University of London, Mile End Road, London E1 4NS, United Kingdom}

\email{z.ghadernezhad@qmul.ac.uk}

\maketitle

\section{Introduction}
\label{SEC:intro}

Let $\Omega$ be a countably infinite set and let $S_\infty$ denote the group of all permutations of $\Omega$. 
For a subgroup $G\leq S_\infty$ and $A\subseteq\Omega$
put
$G_{A}:=\{g\in G: ga=a\ \text{for all } a\in A\}$.

The family $\{G_{A}: A\subseteq\Omega \text{ finite}\}$ is a base of neighbourhoods of the identity for a group topology on $G$, the \emph{point-wise convergence topology}, which we denote by $\tau_{\mathrm {pt}}$. With this topology $S_\infty$ is a \emph{Polish group}: its topology is separable and completely metrizable.

Given a countable first-order structure $M$ with underlying set $\Omega$, the automorphism group $\Aut{(M)}$ is a closed subgroup of $(S_\infty,\tau_{\mathrm {pt}})$ and is therefore itself Polish; conversely, every closed subgroup of $S_\infty$ is the automorphism group of some first-order structure on $\Omega$.

Every open subgroup of $\Aut(M)$ has small
index: it contains $\Aut_{A}(M)$ for some finite $A\subseteq M$, and $[\Aut(M):\Aut_{A}(M)]$ is at most the number of images of $A$ under $\Aut(M)$, which is countable.
We say that $\Aut(M)$ has the \emph{small index property} (SIP) if the converse holds, that is, if for every $H\leq\Aut(M)$ with $[\Aut(M):H]<2^{\aleph_{0}}$ there is a finite
$A\subseteq M$ with $\Aut_{A}(M)\leq H$; equivalently, every subgroup of small index is open in $\tau_{\mathrm {pt}}$. 
Since the open subgroups form a neighbourhood base of the identity, the SIP implies that the point-wise topology on $\Aut(M)$ can be recovered from the abstract group structure of $\Aut(M)$ alone.

When $M$ is $\omega$-categorical (equivalently, when $\Aut(M)$ is \emph{oligomorphic}) and $\Aut(M)$ has the SIP, the abstract group $\Aut(M)$ determines $M$ up to \emph{bi-interpretability} \cite{MR831437}; see \cite[\S5]{MacS} for a survey. The SIP is known for many countable $\omega$-categorical structures, among them the dense linear order $(\mathbb Q,\leq)$, the random graph, and the $\omega$-stable $\omega$-categorical structures --- for instance a pure countably infinite set, or a countably infinite dimensional vector space over a finite or countable division ring; see \cite{LAS-AUT-SIP} for references.

More recently, Paolini and Shelah \cite{PaoShe-SSIP} have shown that for countable homogeneous structures which are algebraically locally finite and carry a free stationary independence relation, the SIP implies the \emph{strong} small index property, and have deduced the strong SIP for countable free homogeneous structures in a locally finite relational language. 
Here $\Aut(M)$ has the \emph{strong} SIP if for every subgroup $H$ of small index there is a finite $A\subseteq M$ with $\Aut_{A}(M)\leq H\leq\Aut_{\{A\}}(M)$, i.e.\ $H$ lies between the point-wise and the set-wise stabiliser of $A$. 
In \cite{PaoShe-Rec} they use the strong SIP, together with local finiteness of algebraicity, to obtain reconstruction up to bi-definability without assuming $\omega$-categoricity. In a different direction, a version of the SIP for homogeneous models inside abstract elementary classes, formulated with respect to a topology generated by stabilisers of \emph{closures} rather than of finite
sets, was obtained in \cite{GhaVil}; the reformulation of our Theorem~A given in Remark~\ref{rem:tau-gcl} below is of that kind.

There are only a few general techniques for establishing the SIP; see \cite{MacS} for an overview. 
The most widely used one rests on a combinatorial condition, the \emph{extension property for partial automorphisms} (EPPA, also called the \emph{Hrushovski property}) for the class of finite substructures of $M$ (Definition~\ref{def-ext}). 
Hrushovski in \cite{HruEP} proved it for the class of all finite graphs: every finite graph $X$ embeds as an induced subgraph into a finite graph $Z$ so that every isomorphism between induced subgraphs of $X$
extends to an automorphism of $Z$. 
This was extended by Herwig \cite{MR1357282,MR1658539}, by Herwig and Lascar \cite{MR1621745}, by Hodkinson and Otto \cite{HodOtt}, and, in a common generalisation of these, by Hubi\v{c}ka, Kone\v{c}n\'y and Ne\v{s}et\v{r}il \cite{HKN22}. 

By \cite{MR1231710}, the SIP for $\Aut(M)$ with $M$ $\omega$-categorical follows from the extension property for the class of finite substructures of $M$ together with an amalgamation condition (Definition~\ref{def:base0}(3)); the abstract framework behind this implication was subsequently isolated by Kechris and Rosendal \cite{MR2308230} in terms of \emph{ample generics}. 

Outside the $\omega$-categorical context much less is known. The SIP has been established for some countable structures that are not saturated, such as free groups of countable rank \cite{MR1438640} and arithmetically saturated models of arithmetic \cite{MR1368465}, and Lascar and Shelah \cite{MR1204064} proved that the automorphism group of every uncountable saturated structure has the SIP.

Later, Lascar established the following result (Théorème 1 in \cite{Lascar}): Suppose $M$ is a countable saturated structure containing a definable (without parameters) strongly minimal subset $D$ such that $M$ lies in the algebraic closure of $D$. Then, for any subgroup $H\leq \Aut(M)$ of countable index, there exists a finite subset $A\subseteq M$ such that every \emph{$A$-strong} automorphism belongs to $H$. We refer to this property as the \emph{almost small index property}, or \emph{almost SIP} for short. In other words, every subgroup of countable index contains the point-wise stabiliser of $\acleq{A}$ for some finite $A\subseteq M$.  
Here $\acleq{A}$ denotes the algebraic closure of $A$ in $M^{\mathrm{eq}}$, that is, the set of imaginary elements (classes of $\emptyset$-definable equivalence relations on tuples of $M$) that have only finitely many conjugates over $A$.

In this paper we combine the method of Hodges, Hodkinson, Lascar and Shelah \cite{MR1231710} with Lascar's technique from \cite{Lascar} and show that the automorphism group of a countable ab-initio (uncollapsed) Hrushovski generic structure has the small index property. Since the class of finite substructures of such a structure does \emph{not} have the extension property (Remark~\ref{rem-EP}), the standard route is unavailable; instead we establish the extension property for a suitable subclass and transfer it.

\subsection*{Outline}
Section~\ref{SEC:intro} fixes the setting and states the four results A--D. Section~\ref{SECTHM} proves Theorem~A, the analogue of Lascar's theorem in the present setting. Section~\ref{SECLEM} contains the (mostly routine) proof of Lemma~B, together with the extension lemma (Lemma~\ref{lem:extend}) that makes the restriction map $\rho_X$ surjective.
Section~\ref{SEC:4} contains the combinatorial heart of the paper: the extension property for the class of structures of pre-dimension zero (Lemma~\ref{lem:ep-im}), and the deduction of Proposition~C.
Section~\ref{SEC:5} discusses the collapsed case, the irrational case, and related open problems.

\subsection*{Setting}\label{subsection-setting}

Hrushovski's constructions \cite{Hrunew} admit many variants and can be set up at various levels of generality. We work throughout with one basic case: this keeps the combinatorial bookkeeping of Section~\ref{SEC:4} manageable, and the generalisations that our argument does and does not cover are discussed in Section~\ref{SEC:5}. For a general reference on these constructions see \cite{Wag1}.

Let $L=\{R\}$, where $R$ is a binary relation symbol, and let $\mathcal K$ be the class of all finite $L$-structures in which $R$ is irreflexive and symmetric; that is, $\mathcal K$ is the class of all finite graphs. 
If $A$ and $B$ are subsets of an $L$-structure $P$, we write $AB$ for the induced $L$-substructure of $P$ with domain $A\cup B$, and $A\subseteq_{\mathrm{fin}}P$ when $A$ is a finite substructure of $P$. We write $R(A)$ for the set of edges of $A$, and
\[
R(A;B):=\bigl\{\{a,b\}: a\in A,\ b\in B,\ R^{AB}(a,b)\bigr\}
\]
for the set of edges of $AB$ running between $A$ and $B$.

Fix an integer $\mathfrak m\geq 2$ and define the \emph{pre-dimension function} $\delta:\mathcal K\to\mathbb Z$ by
\[
\delta(A):=\mathfrak m\cdot|A|-|R(A)|,
\]
for $A\in \mathcal K$.

\begin{rem}\label{rem:coefficients}
Dividing by $\mathfrak m$ turns $\delta$ into $|A|-\alpha|R(A)|$ with $\alpha=1/\mathfrak m$, so the classes considered here are precisely the ab-initio classes whose coefficient is a rational of the form $1/\mathfrak m$. For a general rational coefficient one takes $\delta(A)=q|A|-p|R(A)|$ with $\gcd(p,q)=1$ and $0<p<q$.
Most of what follows is insensitive to this change, but not all of it: in
Case~3 of the proof of Lemma~\ref{lem-sym} one adds $r$ new vertices, each
joined only to $m_{k}$, and the argument needs $rp=\dim{m_{k}/A_{k}}$, which
can always be arranged only when $p=1$. We therefore state our results for
$\delta(A)=\mathfrak m|A|-|R(A)|$ and return to the general rational case, and
to the irrational one, in Section~\ref{SEC:5}.
\end{rem}

For $A,B\in\mathcal K$ we say that $A$ is $\delta$-\emph{self-sufficient}, or
\emph{$\leqslant_\delta$-closed}, in $B$, written $A\leqslant_\delta B$, if $A\subseteq B$ and $\delta(A')\geq\delta(A)$ for every $A'$ with $A\subseteq A'\subseteq B$. 
If $N$ is an infinite $L$-structure, $A\in\mathcal K$ and $A\subseteq N$, we write $A\leqslant_\delta N$ if $A\leqslant_\delta B$ for every $B\subseteq_{\mathrm{fin}}N$ containing $A$.

The following is standard; see for instance \cite{Hrunew} or \cite[\S3.1]{Wag1}.

\begin{lem}\label{lem:basic-generic-delta} 
Let $P$ be an $L$-structure and let $A,B,C,X\subseteq_{\mathrm{fin}} P$. Then
\begin{enumerate}
\item $\delta(AB)\leq\delta(A)+\delta(B)-\delta(A\cap B)$, with equality if and only if $R(A\setminus B;B\setminus A)=\emptyset$;
\item if $A\leqslant_\delta B$ and $X\subseteq B$, then $A\cap X\leqslant_\delta X$;
\item if $A\leqslant_\delta B$ and $B\leqslant_\delta C$, then $A\leqslant_\delta C$.
\end{enumerate}
\end{lem}

Let $\mathcal K_{0}\subseteq\mathcal K$ be the class of all $A\in\mathcal K$ such that $\delta(A')\geq 0$ for every $A'\subseteq A$. The pair $(\mathcal K_{0},\leqslant_\delta)$ is the \emph{ab-initio class} induced by $\delta$.  Since $\delta$ is fixed from now on, we drop the subscript and write $\leqslant$ for $\leqslant_{\delta}$.

For $A,B,C\in\mathcal K$ with $A\subseteq B$ and $A\subseteq C$, the
\emph{free amalgam} $B\otimes_{A}C$ is the disjoint union of $B$ and $C$ over
$A$, carrying no relations other than those of $B$ and of $C$.

\begin{fact}\label{free-amalgam}
$(\mathcal K_{0},\leqslant)$ has the \emph{free amalgamation property}: if $A,B,C\in\mathcal K_{0}$ with $A\leqslant B$ and $A\leqslant C$, then
$B\otimes_{A}C\in\mathcal K_{0}$ and $B\leqslant B\otimes_{A}C$,
$C\leqslant B\otimes_{A}C$. Moreover
$\delta(B\otimes_{A}C)=\delta(B)+\delta(C)-\delta(A)$. 
\end{fact}

By Fact~\ref{free-amalgam} and a Fra\"iss\'e-style construction one obtains the
following; see \cite{Wag1} or \cite[\S2]{Zthesis}.

\begin{thm}[cf.\ {\cite[\S3]{Wag1}}, {\cite[\S2.2]{Zthesis}}]\label{thm-generic}
There is a countably infinite $L$-structure $M$, unique up to isomorphism with
these properties, such that
\begin{enumerate}
\item a finite $L$-structure embeds into $M$ if and only if it belongs to
$\mathcal K_{0}$ (that is, the age of $M$ is $\mathcal K_{0}$);
\item $M=\bigcup_{i\in\omega}A_{i}$ for some chain
$A_{0}\subseteq A_{1}\subseteq\cdots$ of finite substructures with
$A_{i}\leqslant M$ for every $i\in\omega$;
\item \textup{($(\mathcal K_{0},\leqslant)$-genericity)} for every finite
$A\leqslant M$ and every $B\in\mathcal K_{0}$ with $A\leqslant B$, there is an
embedding $f:B\to M$ such that $f\upharpoonright A=\mathrm{id}_{A}$ and
$f[B]\leqslant M$;
\item \textup{($\leqslant$-homogeneity)} every isomorphism between finite $\leqslant$-closed substructures of $M$ extends to an automorphism of $M$.
\end{enumerate}
\end{thm}

The structure $M$ guaranteed by Theorem~\ref{thm-generic} is called the
\emph{$(\mathcal K_{0},\leqslant)$-generic structure}, or the \emph{ab-initio Hrushovski generic structure}. Throughout the paper $\mathfrak m$, $\delta$, $\mathcal K_{0}$ and the $(\mathcal K_{0},\leqslant)$-generic $ M$ are fixed as above, and
$G:=\Aut(M)$.

We recall, for orientation only, that $\Th(M)$ is $\omega$-stable of infinite Morley rank and that $M$ is its countable saturated model \cite[\S4]{Balshi}. 
Saturation uses the fact that every finite subset of $M$ is contained in a finite $\leqslant$-closed subset of $M$. In the more general frameworks of \cite{Balshi} this can fail, and the generic structure need not be $\omega$-saturated. These facts are not used in the proofs below, but they place $M$ within the scope of Lascar's setting (countable
saturated structures), which is the starting point of
Section~\ref{SECTHM}.

For $A\subseteq_{\mathrm{fin}}M$ let $\cl A$ denote the smallest $\leqslant$-closed finite subset of $M$ containing $A$; its existence follows from the fact that $\delta$ takes only non-negative values on $\mathcal K_{0}$, so that $\delta(A')$ is minimised on some finite $A'\supseteq A$, and its uniqueness from Lemma~\ref{lem:basic-generic-delta}(2).
The \emph{dimension} of $A$ is
\[
\dim{A}:=\delta(\cl A),
\]
and for $A,B\subseteq_{\mathrm{fin}}M$ we write $\dim{B/A}:=\dim{AB}-\dim{A}$.

\begin{fact}\label{fact:cl}
Let $A,B\subseteq_{\mathrm{fin}} M$. Then
\begin{enumerate}
\item $\cl A=\acl{A}$; 
\item $\delta(\cl A)\leq\delta(A)$, and $\delta(B)\geq\dim{A}$ for every finite
$B\subseteq M$ with $\cl A\subseteq B$;
\item $\dim{AB}+\dim{A\cap B}\leq\dim{A}+\dim{B}$.
\end{enumerate}
\end{fact}

Here $\acl{-}$ denotes algebraic closure in the model-theoretic sense. Thus Fact~\ref{fact:cl}(1) says that in the present (uncollapsed) setting algebraic closure coincides with self-sufficient closure; in particular $M$ has \emph{locally finite algebraicity}, that is, $\acl{A}$ is finite for every finite $A\subseteq M$.  There is another closure operator, called the \emph{geometric closure}, which can be defined naturally whenever a dimension function is present: for $X\subseteq M$ set
\[
\gcl X:=\{x\in M:\dim{x/A}=0\ \text{for some }A\subseteq_{\mathrm{fin}}X\}.
\]
So $\cl A$ is finite for finite $A$, whereas $\gcl A$ is typically infinite. The operator $\gcl-$ is the one which plays, in Theorem~\ref{thm:one}, the role played by $\acleq{-}$ in Lascar's theorem. 
The dimension function extends to arbitrary $X\subseteq M$ by
\[
\dim{X}:=
\begin{cases}
\max\{\dim{A}:A\subseteq_{\mathrm{fin}}X\}, & \text{if this maximum exists;}\\
\infty, & \text{otherwise.}
\end{cases}
\]
In particular $\dim{\gcl A}=\dim{A}=\delta(\cl A)$ for every finite
$A\subseteq M$.

We recall a notion from \cite[p.~150]{Hrunew}.

\begin{defn}\label{def:0-alg}
Let $A,B\subseteq_{\mathrm{fin}} M$ with $A\cap B=\emptyset$. Then $B$ is
\emph{$0$-algebraic over $A$} if
\begin{enumerate}
\item $\delta(BA)-\delta(A)=0$, and
\item $\delta(B'A)-\delta(A)>0$ for every non-empty proper $B'\subsetneq B$.
\end{enumerate}
$B$ is \emph{$0$-minimally algebraic over $A$} if in addition there is no proper
$A'\subsetneq A$ over which $B$ is $0$-algebraic.
\end{defn}
   These sets are the building blocks of the geometric closure.
If $C\leqslant M$ is finite and $x\in\gcl C$, then a minimal non-empty $B\subseteq\cl{Cx}\setminus C$ with $\delta(BC)=\delta(C)$ is $0$-algebraic
over $C$, and $BC\leqslant M$. Iterating, $\gcl A$ is obtained from $\cl A$
by successively adding $0$-algebraic sets.

\subsection*{Summary of the notation and conventions}\label{setting-notation}

The notation used throughout the paper.

\begin{itemize}
\item $M$ is the $(\mathcal K_{0},\leqslant)$-generic structure given by Theorem~\ref{thm-generic}.
\item $G=\Aut(M)$. For $X\subseteq M$ we write $G_{X}=\Aut_{X}(M)$ for the
point-wise stabiliser of $X$ in $G$, and $G_{\{X\}}$ for its set-wise
stabiliser; $g[X]$ denotes the image of $X$ under $g\in G$.
\item $\tau_{\mathrm {pt}}$ denotes the point-wise convergence topology on $G$,  and
$\tau_{\mathrm {gcl}}$ the group topology on $G$ whose basic neighbourhoods of the identity are the subgroups $G_{\gcl A}$ for $A\subseteq_{\mathrm{fin}}M$ (note that clearly
$\tau_{\mathrm {pt}} \subseteq \tau_{\mathrm {gcl}}$).
\item If $N_{0}\subseteq N_{1}$ are $L$-structures, $g_{0}\in\Aut(N_{0})$ and
$g_{1}\in\Aut(N_{1})$, we write $g_{0}\leqslant g_{1}$ to mean that $g_{1}$
extends $g_{0}$, i.e.\ ${g_{1}}_{\upharpoonright N_{0}}=g_{0}$. The same notation is
used for partial maps.
\item A \emph{partial $\leqslant$-isomorphism} of $M$ is an isomorphism $f:A\to B$ with $A,B\leqslant M$ finite. 
\item $\mathcal  C:=\{A\in\mathcal K_{0}:\delta(A)=0\}$ and
$\hat M:=\gcl\emptyset\subseteq M$. 
\end{itemize}

\subsection*{Main results}

Recall that $\mathcal K_{0}$, $M$ and $G=\Aut(M)$ are fixed as in the previous subsection. 
The main result of this paper is Theorem~\ref{thm-SIP}: the group $G$ has the small index property.

Given the known strategies for proving the SIP, it would have sufficed to show that $(\mathcal K_{0},\leqslant)$ has the extension property
(Definition~\ref{def-ext}). However, by Corollary 5.1.15 of \cite{Zthesis} the class $(\mathcal K_{0},\leqslant)$ does \emph{not} have the extension property, and by Remark~\ref{rem-EP} this failure is robust. We therefore proceed differently, in three steps.

The first step, proved in Section~\ref{SECTHM} by a method close to Lascar's in \cite{Lascar}, is the following; it was suggested in \cite{Zthesis}.

\newcommand{\textfreedomtheorem}[0]{ 
Let $M$ be the $(\mathcal K_{0},\leqslant)$-generic
structure obtained from the pre-dimension function
$\delta(A)=\mathfrak m|A|-|R(A)|$, and let $G=\Aut(M)$. If $H\leq G$ satisfies
$[G:H]<2^{\aleph_{0}}$, then there is a finite $A\subseteq M$ such that
$G_{\gcl A}\leq H$.

    }
    
\begin{thmx}
\label{thm:one} \textfreedomtheorem
\end{thmx}

\begin{rem}\label{rem:tau-gcl}
Theorem~A states precisely that $G$ has the small index property with respect
to the finer topology $\tau_{\mathrm{gcl}}$, i.e.\ every subgroup of small index
is $\tau_{\mathrm{gcl}}$-open. It is the analogue of Lascar's theorem for
$N\models\mathrm{ACF}_{0}$ countable saturated, where $\acleq{A}$ may be replaced
by the field-theoretic algebraic closure of $A$: every subgroup of countable
index contains the automorphisms fixing $\acl{A}$ point-wise, for some finite
$A$. We note also that Th\'eor\`eme~1 of \cite{Lascar} is stated for subgroups
of countable index, whereas the Baire category argument of
Section~\ref{SECTHM} gives the conclusion for all subgroups of index less than
$2^{\aleph_{0}}$.
\end{rem}

The second step lists the properties of the set-wise stabilisers $G_{\{X\}}$, $X=\gcl A$ for $A\subseteq_{\mathrm{fin}} M$, which allow the passage from $\tau_{\mathrm {gcl}}$-openness to $\tau_{\mathrm {pt}}$-openness. It is proved in Section~\ref{SECLEM}; most parts are straightforward, the substantial one being the surjectivity in~(4), which rests on the extension lemma (Lemma~\ref{lem:extend}).

\begin{lemx}
\label{pro:main}
Let $G=\Aut\left(M\right)$. Assume $X = \gcl A$ where $A \subseteq_{\mathrm{fin}} M$. Then:
\begin{enumerate}
\item $G_{\{X\}}$ is an open, hence closed, subgroup of $(G,\tau_{\mathrm {pt}})$; in particular it is a Polish group.
\item $G_{\{X\}}$ has small index in $G$.
\item If $G_{\{X\}}$ has the SIP, then $G$ also has the SIP.
\item The restriction map $\rho_X : G_{\{X\}} \rightarrow \Aut(X)$  defined by $h \mapsto h_{\upharpoonright X}$ is a continuous, open, and surjective homomorphism with $\ker\rho_X=G_X$.
\end{enumerate}
\end{lemx}

The third step supplies the remaining ingredient, the SIP for $\Aut(X)$. Although
$(\mathcal K_{0},\leqslant)$ does not have the extension property, we show in Section~\ref{SEC:4} that a suitable subclass does.

\newcommand{\stmtC}{Let $A\subseteq_{\mathrm{fin}}M$ and $X=\gcl A$. Then $\Aut(X)$ has
the SIP.}
 
\newcommand{\stmtD}{Let $M$ be the $(\mathcal K_{0},\leqslant)$-generic
structure obtained from the pre-dimension function
$\delta(A)=\mathfrak m|A|-|R(A)|$. Then $\Aut(M)$ has the small index property:
every subgroup of index less than $2^{\aleph_{0}}$ is open in the point-wise
convergence topology.}
 
 \begin{propx}\label{thm:main-1}
\stmtC
\end{propx}

The proof, given in Section~\ref{SEC:4}, reduces to the case $A=\emptyset$ and
rests on the extension property for the class of structures of pre-dimension
zero (Lemma~\ref{lem:ep-im}).
The general case is obtained by relativising this argument over $\cl A$.
 
Combining Theorem \ref{thm:one}, Lemma \ref{pro:main} and Proposition \ref{thm:main-1} we obtain our main result. 
 
\begin{thmx}\label{thm-SIP}
\stmtD
\end{thmx}
 
\begin{proof}
Let $H\leq G$ with $[G:H]<2^{\aleph_{0}}$. By Theorem~A there is a finite
$A\subseteq M$ with $G_{X}\leq H$, where $X=\gcl A$. Put $H':=H\cap G_{\{X\}}$.
Then $G_{X}\leq H'$, and $[G_{\{X\}}:H']\leq[G:H]<2^{\aleph_{0}}$.
 
Consider the restriction map $\rho_{X}:G_{\{X\}}\to\Aut(X)$ of Lemma~B(4).
Since $\rho_{X}$ is surjective, the cosets of $H'$ in $G_{\{X\}}$ map onto the
cosets of $\rho_{X}(H')$ in $\Aut(X)$, so
$[\Aut(X):\rho_{X}(H')]\leq[G_{\{X\}}:H']<2^{\aleph_{0}}$; that is,
$\rho_{X}(H')$ has small index in $\Aut(X)$. 
By Proposition~C the group $\Aut(X)$ has the SIP, so $\rho_{X}(H')$ is open in
$\Aut(X)$, and by continuity of $\rho_{X}$ the set
$\rho_{X}^{-1}(\rho_{X}(H'))=H'\ker(\rho_{X})$ is open in $G_{\{X\}}$. By
Lemma~B(4) we have $\ker(\rho_{X})=G_{X}\leq H'$, whence
$\rho_{X}^{-1}(\rho_{X}(H'))=H'$, and therefore $H'$ is open in $G_{\{X\}}$.
By Lemma~B(1) the subgroup $G_{\{X\}}$ is open in $G$, so $H'$ is open in $G$;
finally $H$ contains the open subgroup $H'$ and is therefore a union of cosets
of $H'$, hence open in $G$. 
\end{proof}
 
\begin{rem}\label{rem:projection} 
The proof above used only that $\rho_{X}$ is a continuous, open surjection. The
same observation, applied in the opposite direction, is worth recording: if
$\pi:G_{1}\to G_{2}$ is a continuous, open and surjective homomorphism of
topological groups and $G_{1}$ has the SIP, then so does $G_{2}$. Indeed, if
$[G_{2}:K]<2^{\aleph_{0}}$ then $[G_{1}:\pi^{-1}(K)]=[G_{2}:K]$, so
$\pi^{-1}(K)$ is open in $G_{1}$ and hence $K=\pi(\pi^{-1}(K))$ is open in
$G_{2}$. Together with Lemma~B(4) this shows that Proposition~C is a
consequence of Theorem~D, and so that the two are equivalent. The implication
runs from $M$ to $\gcl A$ and not conversely, so it does not shorten the proof
of Theorem~D.
\end{rem}

\section{Small index subgroups are $\tau_{\rm gcl}$-open}
\label{SECTHM}
 We begin by recalling a few notions from descriptive set theory; further details may be found in \cite{MR1425877}. Let $X$ be a topological space. A subset $A\subseteq X$ is \emph{nowhere dense} if the interior of its closure is empty; a countable union of nowhere dense sets is \emph{meager}, and the complement of a meager set is \emph{comeager}. A \emph{topological group} is a group with a topology making multiplication and inversion continuous; such a topology is determined by the filter of neighbourhoods of the identity. A \emph{Polish space} is a separable, completely metrizable topological space, and a topological group is \emph{Polish} if its topology is Polish. Since Polish groups are Baire spaces, a meager subgroup of a Polish group has index $2^{\aleph_{0}}$; we use this in the form: \emph{a subgroup of small index of a Polish group is non-meager}. 

Automorphism groups of countable structures are Polish in the point-wise convergence topology $\tau_{\mathrm {pt}}$. In the proof of Theorem~A, however, we shall work with a subgroup of $G$ carrying a Polish group topology strictly finer than the one induced by $\tau_{\mathrm {pt}}$; note that $G$ itself is not Polish in the finer topology $\tau_{\mathrm {gcl}}$. See the discussion preceding Proposition~\ref{prop:polish-GS}. 

The following definition is adapted from \cite{LAS-AUT-SIP} and
\cite{MR1231710}, with some terminology added for readability. It is used both here and in Section~\ref{SEC:4}. 

\begin{defn}\label{def:base0}
Let $N$ be a countable structure, let $\Gamma\leq\Aut(N)$, and let $\tau$ be a group topology on $\Gamma$ finer than the topology induced by $\tau_{\mathrm {pt}}$.
\begin{enumerate}
\item A $\Gamma$-invariant family $\mathscr B\subseteq\mathcal P(N)$ is
\emph{cofinal} if every finite subset of $N$ is contained in some member of
$\mathscr B$.
\item A countable, $\Gamma$-invariant, cofinal family $\mathscr B$ of substructures of $N$ is a \emph{base} with respect to $\tau$ if $\Gamma_{A}$ is
$\tau$-open for every $A\in\mathscr B$.
\item A base $\mathscr B$ is an \emph{amalgamation base} if:
\begin{enumerate}
\item whenever $A\in\mathscr B$ and $e_{1},\dots,e_{n}$ are restrictions of elements of $\Gamma$ to members of $\mathscr B$, with $\dom{e_{i}}\cup\ran{e_{i}}\subseteq A$, there are $B\in\mathscr B$ with $A\subseteq B$ and $f_{i}\in\Aut(B)$ with $e_{i}\leqslant f_{i}$ for $1\leq i\leq n$;
\item whenever $A,B,C\in\mathscr B$ with $A\subseteq B$ and $A\subseteq C$, there is $\alpha\in\Gamma_{A}$ with $\alpha[B]\cap C=A$ such that for all $g\in\Aut(\alpha[B])$ and $h\in\Aut(C)$ which are restrictions of elements of $\Gamma$ and satisfy $g_{\upharpoonright A}=h_{\upharpoonright A}$, the map $g\cup h$ extends to an element of $\Gamma$.
\end{enumerate}
\item Let $\mathscr B$ be a base, $n\geq 1$, and
$\gamma=(g_{1},\dots,g_{n})\in\Gamma^{n}$. Then $\gamma$ is
\emph{$\mathscr B$-generic} if:
\begin{enumerate}
\item for every $A\in\mathscr B$ the family
$\{B\in\mathscr B: A\subseteq B,\ g_{i}[B]=B\ \text{for }1\leq i\leq n\}$ is cofinal in $\mathscr B$; equivalently, the corresponding stabilisers $\Gamma_{B}$ form a base of $\tau$-open neighbourhoods of $1$;
\item whenever $A\in\mathscr B$ is such that $\gamma_{\upharpoonright A}$ is a
sequence of automorphisms of $A$, $B\in\mathscr B$ contains $A$, and
$\theta=(t_{i}:1\leq i\leq n)$ is a sequence of automorphisms of $B$ with
${g_{i}}_{\upharpoonright A}\leqslant t_{i}$ for all $i$, then there is
$\alpha\in\Gamma_{A}$ with $\gamma^{\alpha}=(\alpha g_{i}\alpha^{-1}:1\leq
i\leq n)$ extending $\theta$.
\end{enumerate}
\end{enumerate}
\end{defn}

\begin{rem}\label{rem:KR}
Definition~\ref{def:base0}(3) is the notion used in \cite{MR1231710}; the
now-standard reformulation in terms of the joint embedding and (cofinal)
amalgamation properties for systems of partial automorphisms, and of
\emph{ample generics}, is due to Kechris and Rosendal \cite{MR2308230}. Our
$\mathscr B$-generic sequences are their generic $n$-tuples, and
Proposition~\ref{prop:conjugate} below is the corresponding uniqueness
statement.
\end{rem}

The next two propositions are the main ingredients of the proof of
Theorem \ref{thm:one}. Both are proved exactly as in the references cited; the only change is that $\tau$ is not assumed to be $\tau_{\mathrm {pt}}$,  which affects neither argument.

\begin{prop}[{\cite[Prop.~2.3]{MR1231710}}, {\cite[Prop.~8]{Lascar}}]
\label{prop:conjugate}
Let $N$, $\Gamma$ and $\tau$ be as in Definition~\ref{def:base0}, and let
$\mathscr B$ be a base. Let $n\geq 1$ and let $\gamma,\gamma'\in\Gamma^{n}$ be
$\mathscr B$-generic. If $\gamma_{\upharpoonright B}=\gamma'_{\upharpoonright B}$ for
some $B\in\mathscr B$ such that $\gamma_{\upharpoonright B}$ is a sequence of
automorphisms of $B$, then $\gamma'=\gamma^{\alpha}$ for some
$\alpha\in\Gamma_{B}$.
\end{prop}

\begin{prop}[{\cite[Prop.~7]{Lascar}}]\label{prop:Gdelta}
Let $N$, $\Gamma$ and $\tau$ be as in Definition~\ref{def:base0}, and let
$\mathscr B$ be an amalgamation base. Then for every $n\geq 1$ the set of
$\mathscr B$-generic elements of $\Gamma^{n}$ is a dense $G_{\delta}$ subset of
$\Gamma^{n}$, where $\Gamma^{n}$ carries the product topology induced by
$\tau$.
\end{prop}

\subsection*{A symmetric basis for $M$}

Recall the setting and notation for $M$ summarised at the end of Section~\ref{SEC:intro}. Throughout, we call a finite $A\subseteq M$ \emph{free} if $A\leqslant M$ and $R(A)=\emptyset$; equivalently, $\dim A=\delta(A)=\mathfrak m|A|$.
\begin{lem}\label{lem-sym}
There is a countable subset $B=\{b_{i}:i\in\omega\}\subseteq M$ such that:
\begin{enumerate}
\item $\gcl B=M$;
\item every $B_{0}\subseteq_{\mathrm{fin}}B$ is free, that is, $B_{0}\leqslant M$ and $R(B_{0})=\emptyset$;
\item every permutation of $B$ extends to an automorphism of $M$.
\end{enumerate}
\end{lem}

\begin{proof}
 Fix an enumeration $\langle m_{i}:i\in\omega\rangle$ of $M$. We construct an increasing chain $A_0\subseteq A_1\subseteq \cdots$ of finite subsets of $M$ such that for every $i\in \omega$

\begin{enumerate}
\item[(a)] $A_i$ is free;
\item[(b)] $m_{i}\in\gcl{A_{i+1}}$.
\end{enumerate}

Setting $B:=\bigcup_{i\in \omega}A_i$ gives $\gcl{B}=M$ by (b), and (2) by (a), since a subset of a free set is free.

Start by choosing $a\in M$ such that $\{a\}\leqslant M$ and $A_0:=\{a\}$. Suppose $A_k$ has been constructed for $k\geq 0$.  For $m_k$ we have the following cases:

\noindent
\emph{Case 1: $\dim{m_{k}/A_k}=0$.} Then $m_{k}\in\gcl{A_k}$ and we may take $A_{k+1}=A_k$. 

\smallskip\noindent
\emph{Case 2: $\dim{m_{k}/A_k}=\mathfrak m$.} Then
\[
\dim{m_{k}A_k}\ \leq\ \delta(m_{k}A_k)\ \leq\ \delta(m_{k})+\delta(A_k)
\ =\ \mathfrak m+\dim{A_k}\ =\ \dim{m_{k}A_k},
\]
so $\delta(m_{k}A_k)=\dim{m_{k}A_k}$ and hence $m_{k}A_k\leqslant M$
and $R(m_{k};A_k)=\emptyset$. Put $A_{k+1}:=\{m_{k}\}\cup A_k$. 

\smallskip\noindent
 
\smallskip\noindent
\emph{Case 3: $\dim{m_{k}/A_k}=l$ with $1\leq l<\mathfrak m$.} Let $B':=\cl{m_{k}A_k}$.  Let $C$ be the structure obtained from $B'$ by adjoining $l$ new points
$c_{1},\dots,c_{l}$, each joined to $m_{k}$ by an edge and to nothing else. Let $\bar c=c_1\cdots c_l$. Then $\delta(C)=\delta(B')+\mathfrak ml-l$ and $\delta(A_kc_1\cdots c_p)= \delta(A_k)+\mathfrak mp$ for every $1\leq p\leq l$. Moreover $C\in\mathcal K_0$, since for $Y\subseteq C$ we have $\delta(Y)\geq\delta(Y\cap B')+(\mathfrak m-1)|Y\cap\bar c|\geq 0$, and the same estimate gives $B'\leqslant C$.  By Theorem~\ref{thm-generic}(3) we may $\leqslant$-embed $C$ over $B'$ into $M$; denote the image again by $C$. Now set $A_{k+1}:=A_k\bar c$. First note that $A_{k+1}$ is free and $A_k\bar c\leqslant C$.
 Note $B'\leqslant M$, and set $B^\star:=B'\setminus A_k$ and let $U\subseteq B^\star$.
Since $A_k\leqslant M$ we have $\delta(UA_k)-\delta(A_k)\geq 0$ for all $U\subseteq B^\star$. In particular if $m_{k}\in U$ then $A_kU\supseteq m_kA_k$, so $\delta(A_kU)\geq\dim{A_kU}\geq\dim{m_kA_k}=\delta(B')=\delta(A_k)+l$. For checking $A_k\bar c\leqslant C$: consider $A_k\bar c\subseteq F\subseteq C$ and put $U:=F\cap B^\star$. Each $c_{i}\in \bar c$ has its unique edge to $m_{k}$, so
\[
\delta(F)=\delta(A_kU)+\mathfrak ml-\varepsilon,
\]
where $\varepsilon := |R(F\setminus\bar c;\bar c)|$. If $m_k\in U$ then $\varepsilon=l$ otherwise $\varepsilon=0$. So in either case $\delta(F)\geq\delta(A_k)+\mathfrak ml=\delta(A_k\bar c)$. As $C\leqslant M$, Lemma~\ref{lem:basic-generic-delta}(3) gives $A_{k+1}\leqslant M$; since the $c_{i}$ have no edges to $A_{k}$ or to each other, $A_{k+1}$ is free.
Second $\dim{A_{k+1}}=\delta(A_k)+\mathfrak ml=\delta(C)=\dim{C}$
which implies $C\subseteq \gcl{A_{k+1}}$, and in particular $m_k\in \gcl{A_{k+1}}$.

Finally, (3) is a consequence of Lemma~\ref{lem:coherent}(i) below, whose proof uses only (1) and (2).

\end{proof}
\subsection*{Stationary independence}

Let $\mathcal X:=\{\gcl A: A\subseteq_{\mathrm{fin}}M\}$ and let $\mathcal F$ be the set
of isomorphisms $f:X\to Y$ with $X,Y\in\mathcal X$ that extend to automorphisms
of $M$. By \cite[Lemma 4.3, Cor.~4.8]{ZEvansTent} the independence relation
derived from $\gcl-$ is a stationary independence relation in the sense of
\cite[Def.~2.2]{ZEvansTent}, compatible with $\mathcal X$. For a cofinal
$\mathcal Y\subseteq\mathcal X$ and $\mathcal S\subseteq\mathcal F$ put
\[
G_{\mathcal Y}(\mathcal S):=\{g\in G: g\upharpoonright X\in\mathcal S
\ \text{for all } X\in\mathcal Y\}.
\]

\begin{prop}[{\cite[Lemma 2.3]{ZEvansTent}}]\label{prop:polish-GS}
	Let $\mathcal Y\subseteq\mathcal X$ be cofinal. For every countable
	$\mathcal S_{0}\subseteq\mathcal F$ there is a countable $\mathcal S$ with
	$\mathcal S_{0}\subseteq\mathcal S\subseteq\mathcal F$ such that
	$G_{\mathcal Y}(\mathcal S)$, topologised by declaring the sets
	$\mathcal O(f)=\{g: f\leqslant g\}$ for $f\in\mathcal S$ to be basic open, is a
	Polish group.
\end{prop}

We write $\tau_{\mathcal S}$ for this topology; it refines the topology induced
by $\tau_{\mathrm {pt}}$.

\subsection*{Coherence}
Let $B=\{b_{i}:i\in\omega\}$ be the set given by Lemma~\ref{lem-sym}.
Write $\Sym(B)$ for the group of all permutations of $B$ and, for $B_{0}\subseteq B$, $\Sym_{B_{0}}(B)$ for the subgroup of those fixing $B_{0}$ point-wise.

The following lemma is used repeatedly below.

\begin{lem}\label{lem:rigid}
Let $B_{0}\subseteq_{\mathrm{fin}}B$ and $X:=\gcl{B_{0}}$.
For every $b\in B\setminus B_{0}$ we have $\dim{b/X}=\mathfrak m$. In
particular $X\cap B=B_{0}$ and $R(X;B\setminus B_{0})=\emptyset$.
\end{lem}
\begin{proof}
	Let $b\in B\setminus B_{0}$ and let $Y$ be finite with
	$B_{0}\subseteq Y\leqslant X$. By Lemma~\ref{lem-sym}(2) the set $B_{0}b$ is free,
	so $\dim{B_{0}b}=\dim{B_{0}}+\mathfrak m$. Since $B_{0}\subseteq Y\subseteq X=\gcl{B_{0}}$
	we have $\dim Y=\dim{B_{0}}$, whence
	$\dim Y+\mathfrak m=\dim{B_{0}b}\leq\dim{Yb}\leq\dim Y+\mathfrak m$. So
	$\dim{b/Y}=\mathfrak m$ and therefore
	$\delta(Yb)\geq\dim{Yb}=\delta(Y)+\mathfrak m$, which forces
	$R(b;Y)=\emptyset$. As $Y$ was arbitrary, $\dim{b/X}=\mathfrak m$ and there is no
	edge between $b$ and $X$; in particular $b\notin X$, so $X\cap B=B_{0}$.

\end{proof}

\begin{lem}\label{lem:coherent}
	There is a group homomorphism $s:\Sym(B)\to\Aut(M)$ such that
	\begin{enumerate}
		\item[(i)] $s(\beta)_{\upharpoonright B}=\beta$ for every $\beta\in\Sym(B)$;
		\item[(ii)] $s(\beta)\in\Aut_{\gcl{B_{0}}}(M)$ whenever $B_{0}\subseteq_{\mathrm{fin}}B$
		and $\beta\in\Sym_{B_0}(B)$.
	\end{enumerate}
\end{lem}
\begin{proof}
Write $S,S',T,T'$ for finite subsets of $B$. Note that each is free, so
$\dim S=\mathfrak m|S|$. We use three facts.
\begin{enumerate}
\item[(a)] If $Z$ is finite with $S\subseteq Z\subseteq\gcl S$, then
$\dim Z=\mathfrak m|S|$, so $Z\leqslant M$ if and only if
$\delta(Z)=\mathfrak m|S|$. Hence, if $X,Y$ are finite and $\leqslant$-closed with
$S\subseteq X,Y\subseteq\gcl S$, then $X\cup Y\leqslant M$ and
$R(X\setminus Y; Y\setminus X)=\emptyset$.

\item[(b)] $\gcl T\cap\gcl{T'}=\gcl{T\cap T'}$. Indeed, if
$x\in\gcl T\cap\gcl{T'}$, then Fact~\ref{fact:cl}(3) applied to $Tx$ and $T'x$
gives
$\dim{TT'}+\dim{(T\cap T')x}\leq\dim T+\dim{T'}=\dim{TT'}+\dim{T\cap T'}$,
so $\dim{x/T\cap T'}=0$.
\item[(c)] Put $U_{S}:=\bigcup_{T\subsetneq S}\gcl T$. Every edge inside
$U_{S}\cup S$ lies in some $\gcl T$ with $T\subsetneq S$. Indeed, suppose
$x\in\gcl T$ and $y\in\gcl{T'}$, with $T,T'\subsetneq S$, are joined by an
edge, but $x\notin\gcl{T'}$ and $y\notin\gcl T$. By
Lemma~\ref{lem:rigid}, neither $x$ nor $y$ lies in $B$; moreover there are
no edges between $S\setminus T$ and $\cl{Tx}$, so
$\delta(\cl{Tx}\cup S)=\mathfrak m|S|$ and $\cl{Tx}\cup S$ is $\leqslant$-closed
by (a), and similarly for $\cl{T'y}\cup S$. Then $x$ lies only in the first of
these sets and $y$ only in the second, contradicting~(a). For $|S|=1$, there
is also no edge between $S$ and $U_{S}=\gcl\emptyset$, by
Lemma~\ref{lem:rigid}.
\end{enumerate}
 
It suffices to find, for every bijection $\beta:S\to S'$, an isomorphism
$\theta_{\beta}:\gcl S\to\gcl{S'}$ extending $\beta$, with
$\theta_{\mathrm{id}_S}=\mathrm{id}$,
$\theta_{\gamma\beta}=\theta_{\gamma}\theta_{\beta}$ and
$\theta_{\beta_{\upharpoonright T}}={\theta_{\beta}}_{\upharpoonright\gcl T}$ for
$T\subseteq S$. Indeed, $s(\beta):=\bigcup_{S}\theta_{\beta_{\upharpoonright S}}$
is then well defined by (b) and the last condition; it is an automorphism
because $M=\gcl B$ and any two points of $M$ lie in a common $\gcl S$; and it
is a homomorphism satisfying (i) and (ii), where for (ii) note that
$s(\beta)_{\upharpoonright\gcl{B_0}}=\theta_{\mathrm{id}_{B_0}}=\mathrm{id}$.
 
We construct the $\theta_{\beta}$ by induction on $n=|S|$. For $n=0$ put
$\theta_{\mathrm{id}_\emptyset}:=\mathrm{id}_{\gcl\emptyset}$. Let $n\geq 1$, suppose
the $\theta_{\beta}$ are defined for all bijections between sets of size $<n$,
and let $\beta:S\to S'$ with $|S|=n$. Every $T\subsetneq S$ has size $<n$, so
$\theta_{\beta_{\upharpoonright T}}:\gcl T\to\gcl{\beta[T]}$ is already
defined. Let
\[
\rho(\beta):=\beta\cup\bigcup_{T\subsetneq S}\theta_{\beta_{\upharpoonright T}}
\ :\ U_{S}\cup S\to U_{S'}\cup S'.
\]
This is well defined: a point of $\gcl T\cap\gcl{T'}$ lies in
$\gcl{T\cap T'}$ by (b), where both maps agree with
$\theta_{\beta_{\upharpoonright T\cap T'}}$. It is an isomorphism, since by
(c) every edge lies inside a single $\gcl T$. Moreover
$\rho(\mathrm{id}_S)=\mathrm{id}$ and $\rho(\gamma\beta)=\rho(\gamma)\rho(\beta)$.
It remains to extend each $\rho(\beta)$ from $U_{S}\cup S$ to all of
$\gcl S$, keeping these properties.
 
Write $V^{S}$ for the current domain of the
maps $\rho(\beta)$ with source $S$; initially $V^{S}=U_{S}\cup S$, and
$\rho(\beta)$ maps $V^{S}$ onto $V^{S'}$. We keep the conditions that
$V^{S}\setminus U_{S}$ is finite and that every finite part of $V^{S}$ lies in
a finite closed $C$ with $S\subseteq C\subseteq V^{S}$. Initially the second
condition holds, since each $z\in\gcl T$ lies in the closed set
$\cl{Tz}\cup S$ (see the proof of (c)), and finite unions of such sets are
closed by (a). Also by (a), the $\rho(\beta)$ map such closed sets $C$ to
closed sets.
 
Let $x\in\gcl S\setminus V^{S}$, and take such a $C$ containing
$V^{S}\cap\cl{Sx}$. Then $D:=\bigcup_{g\in\Sym(S)}\rho(g)[C]$ is closed by (a)
and $\rho[\Sym(S)]$-invariant, so $\rho(\beta)[D]$ depends only on the target
$S'$ of $\beta$. As $\cl{Sx}\subseteq\gcl S$, by (a) the set $D\cup\cl{Sx}$ is
closed, so $\cl{Dx}\subseteq D\cup\cl{Sx}$. Hence
$E:=\cl{Dx}\setminus D$ contains $x$ and is disjoint from $V^{S}$, since
$E\cap V^{S}\subseteq C\setminus D=\emptyset$. In particular $E\cap U_S=\emptyset$, so
$\dim{z/T}>0$ for $z\in E$ and $T\subsetneq S$.
 
We put $E_{\mathrm{id}_S}:=E$ and $\psi_{\mathrm{id}_S}:=\mathrm{id}_{DE}$.
For every other bijection $\beta:S\to S'$, $\leqslant$-genericity
(Theorem~\ref{thm-generic}(3) with Fact~1.3) gives a copy $E_{\beta}$ of $E$
over the closed set $\rho(\beta)[D]$, via an isomorphism
$\psi_{\beta}:DE\to\rho(\beta)[D]E_{\beta}$ extending
$\rho(\beta)_{\upharpoonright D}$, with $\rho(\beta)[D]E_{\beta}\leqslant M$;
we choose $E_{\beta}$ disjoint from the finitely many points of
$V^{S'}\setminus U_{S'}$ and from the copies already chosen. As
$\delta(\rho(\beta)[D]E_{\beta})=\delta(D)$, we have
$E_{\beta}\subseteq\gcl{S'}$. Since $\psi_{\beta}$ is an isomorphism between
closed sets, for $z\in E$ and $T\subsetneq S$ we have
$\dim{\psi_{\beta}(z)/\beta[T]}=\dim{z/T}>0$, and $\beta[T]$ ranges over the
proper subsets of $S'$; so $E_{\beta}$ also misses $U_{S'}$, and hence
$V^{S'}$. By (a), the only edges from $E_{\beta}$ to $V^{S'}$ and to the other
copies go to $\rho(\beta)[D]$. Hence we may add each $E_{\beta}$ to $V^{S'}$
and set $\rho(\gamma):=\psi_{\gamma\beta}\psi_{\beta}^{-1}$ on $E_{\beta}$;
on $\rho(\beta)[D]$ this agrees with the old $\rho(\gamma)$, so the new maps
are again isomorphisms with $\rho(\mathrm{id}_S)=\mathrm{id}$ and
$\rho(\gamma\beta)=\rho(\gamma)\rho(\beta)$. Only finitely many points are
added to each $V^{S'}$, and the second condition is kept by (a).
 
Treating every pair $(S,x)$ with $|S|=n$ and $x\in\gcl S$ in turn, the
domains exhaust $\gcl S$, and $\theta_{\beta}:=\rho(\beta)$ is as required.
\end{proof}

\begin{proof}[Proof of Theorem \ref{thm:one}]
 Suppose $B$ is a countable set that is obtained from Lemma \ref{lem-sym}.
Our aim is to enrich the language $L$ to $L^\star:=L\cup\mathfrak{F}\cup\left\{ I\right\} $
where $\mathfrak{F}$ is a countable set of functions and $I$
is a unary predicate such that: 
\begin{enumerate}
\item[(E1)] $\left\langle B_{0}\right\rangle _{\mathfrak{F}}=\gcl{B_{0}}$ for
all $B_{0}\subseteq B$;
\item[(E2)] $\mathfrak{F}$ is \emph{compatible} with permutations of $B$, i.e.\ for
each permutation $\beta$ of $B$ there is a unique $\check{\beta}\in\Aut(M^\star)$
such that $\beta\leqslant\check{\beta}$ and
$\check{\beta}\bigl(\langle B_{0}\rangle_{\mathfrak{F}}\bigr)=\langle\beta(B_{0})\rangle_{\mathfrak{F}}$
for all $B_{0}\subseteq_{\mathrm{fin}}B$;
\item[(E3)] $I(M^\star)=B$.
\end{enumerate}
First suppose such an enrichment $L^\star$ of $L$
exists. Let $M^\star$ be the structure $M$ in the
expanded language $L^\star$. Clearly $\Aut{\left(M^\star\right)}$
is a closed subgroup of $\Aut{(M)}$. Assume $H$ is a subgroup
of small index in $\Aut{(M)}$. Then $H\cap\Aut{\left(M^\star\right)}$
has small index in $\Aut{\left(M^\star\right)}$. 
By (E3) every automorphism of $M^\star$ preserves $B$ setwise, so restriction
gives a homomorphism $\rho:\Aut(M^\star)\to\Sym(B)$. It is surjective by (E2)
and injective by (E1), since $M=\gcl{B}$; hence $\Aut(M^\star)\cong\Sym(B)$.

As $B$ is infinite $\Sym(B)\cong S_\infty$ which

by the result of Dixon,
Neumann and Thomas in \cite{MR859950}  has the small index property.
Hence, there is some $B_{0}\subseteq_{\mathrm{fin}}B$ such that $\Aut_{B_{0}}{\left(M^\star\right)}\leq H\cap\Aut{\left(M^\star\right)}$.

We now show that $\Aut_{\hat B_{0}}(M)\leq H$, where $\hat B_{0}:=\gcl{B_{0}}$;
this proves the theorem, with $A:=B_{0}$. 
Let $h\in\Aut_{\hat B_{0}}(M)$ and put
$\mathcal X_{B_{0}}:=\{X\in\mathcal X: B_{0}\subseteq X\}$, which is a cofinal
subfamily of $\mathcal X$. Using Proposition~\ref{prop:polish-GS}, choose a
countable $\mathcal P\subseteq\mathcal F$ containing the identity map on every
member of $\mathcal X_{B_{0}}$, closed under inverses, compositions and
restrictions to members of $\mathcal X_{B_{0}}$, such that
$\hat G:=G_{\mathcal X_{B_{0}}}(\mathcal P)$ is Polish in $\tau_{\mathcal P}$,
and such that:
\begin{enumerate}
\item[(P1)] $\mathrm{id}_{\hat B_{0}}\leqslant\xi$ for every $\xi\in\mathcal P$;
\item[(P2)] $h_{\upharpoonright X}\in\mathcal P$ for every $X\in\mathcal X_{B_{0}}$;
\item[(P3)] for every finite $B_{1}$ with $B_{0}\subseteq B_{1}\subseteq B$ and every
injection $u:B_{1}\to B$ which is the identity on $B_{0}$, the unique
$L^\star$-extension of $u$ to $\gcl{B_{1}}$ belongs to $\mathcal P$;
\item[(P4)] if $\xi,\nu\in\mathcal P$ and $\xi,\nu\leqslant\sigma$ for some
$\sigma\in\Aut(M)$, then $\xi,\nu\leqslant\hat\sigma$ for some
$\hat\sigma\in\mathcal P$;
\item[(P5)] if $\xi:X\to Y$ lies in $\mathcal P$ and $Z\in\mathcal X_{B_{0}}$ with
$X\cup Y\subseteq Z$, then $\xi\leqslant\lambda$ for some
$\lambda\in\mathcal P$ with $\lambda:Z\to Z$;
\item[(P6)] every $\xi\in\mathcal P$ extends to an element of $\hat G$.
\end{enumerate}

Put $K:=\Aut_{B_{0}}(M^\star)$. By (P2) we have $h\in\hat G$; by (P3), $K$ embeds
as a closed subgroup of $\hat G$, which we again denote by $K$; and $K\leq H$
by the previous paragraph.

Conditions (P4)--(P6) say that $\mathcal X_{B_{0}}$ is an amalgamation base
for $(\hat G,\tau_{\mathcal P})$ in the sense of Definition~\ref{def:base0}(3):
(P5) and (P6) give clause (a), and (P4) and (P6), together with stationarity of
the independence relation derived from $\gcl-$, give clause (b). By (1) every
element of $\hat G$ fixes $\hat B_{0}\in\mathcal X_{B_{0}}$ point-wise. Hence,
for every $n\geq1$, Proposition~\ref{prop:Gdelta} shows that the
$\mathcal X_{B_{0}}$-generic $n$-tuples form a dense $G_{\delta}$ subset of
$\hat G^{n}$, and Proposition~\ref{prop:conjugate}, applied with
$B=\hat B_{0}$, shows that any two of them are conjugate in $\hat G$. Thus
$\hat G$ has ample generics in the sense of \cite{MR2308230}, and therefore
the small index property \cite[Thm.~1.6]{MR2308230}.

Since $[\hat G:H\cap\hat G]\leq[G:H]<2^{\aleph_{0}}$, the subgroup $H\cap\hat G$
is open in $\tau_{\mathcal P}$, so it contains a basic neighbourhood
$\mathcal O(\mathrm{id}_{X_{1}})=\hat G_{X_{1}}$ of the identity, with
$X_{1}\in\mathcal X_{B_{0}}$. Since $M=\gcl B$, we may enlarge $X_{1}$ and
assume $X_{1}=\gcl{B_{1}}$ with $B_{0}\subseteq B_{1}\subseteq_{\mathrm{fin}}B$.
Choose a finite $B_{3}\subseteq B$ with $B_{1}\subseteq B_{3}$ and
$h[X_{1}]\subseteq\gcl{B_{3}}$, and $B_{2}\subseteq B\setminus B_{3}$ with
$|B_{2}|=|B_{1}\setminus B_{0}|$. Let $\beta\in\Sym_{B_{0}}(B)$ exchange
$B_{1}\setminus B_{0}$ with $B_{2}$ and fix all other elements, put
$k:=\check\beta\in K$ and $X_{2}:=k[X_{1}]=\gcl{B_{0}B_{2}}$. As $K\leq H$, we
get $\hat G_{X_{2}}=k\hat G_{X_{1}}k^{-1}\leq H$.

Since $B_{0}B_{2}B_{3}$ is free,
$\dim{B_{2}/B_{3}}=\dim{B_{2}/B_{0}}=\mathfrak m|B_{2}|$, so $X_{2}$ is
independent from $\gcl{B_{3}}\supseteq X_{1}\cup h[X_{1}]$ over $\hat B_{0}$.
As $h$ fixes $\hat B_{0}$ point-wise, stationarity shows that
$h_{\upharpoonright X_{1}}\cup\mathrm{id}_{X_{2}}$ extends to an automorphism of
$M$. By (P2), (P4) and (P6) there is $g\in\hat G$ extending
$h_{\upharpoonright X_{1}}\cup\mathrm{id}_{X_{2}}$. Then
$g\in\hat G_{X_{2}}\leq H$ and $g^{-1}h\in\hat G_{X_{1}}\leq H$, so
$h=g\,(g^{-1}h)\in H$.

It remains to construct the enrichment of the language $L$. Let $s:\Sym(B)\to\Aut(M)$ be as in
Lemma~\ref{lem:coherent}, and interpret $I$ as $B$.

For each $n\in\omega$, each tuple $\bar b$ of $n$ distinct elements of $B$ and
each $a\in\gcl{\bar b}$, introduce an $n$-ary function symbol $f_{a,\bar b}$,
interpreted in $M$ by
\[
f_{a,\bar b}(\bar x):=
\begin{cases}
s(\beta)(a), & \text{if } \bar x=\beta\bar b \text{ for some }
\beta\in\Sym(B);\\[2pt]
x_{0}, & \text{otherwise,}
\end{cases}
\]
and let $\mathfrak F$ be the set of all $f_{a,\bar b}$. Since $B$ is countable and each $\gcl{\bar b}$ is countable, $\mathfrak F$ is countable.

This is well defined. The first clause applies precisely when $\bar x$ is a
tuple of $n$ distinct elements of $B$, because $\Sym(B)$ is transitive on such
tuples; in particular the two clauses are exhaustive and mutually exclusive,
and the condition distinguishing them is preserved by every automorphism of $M$
carrying $B$ to $B$. Moreover, if $\beta\bar b=\beta'\bar b$ then
$\beta^{-1}\beta'$ fixes $\bar b$ point-wise, so by
Lemma~\ref{lem:coherent}(ii) it fixes $a\in \gcl{\bar b}$, whence
$s(\beta')(a)=s(\beta)(a)$.

\emph{Condition (E1).} Let $B_{0}\subseteq B$. 
By Lemma ~\ref{lem:rigid}, $\gcl{B_0}\cap B=B_0$. Hence if $\bar x$ is a
tuple from $\gcl{B_{0}}$ to which the first clause applies, then $\bar x$ is a
tuple from $B_{0}$ and $f_{a,\bar b}(\bar x)\in\gcl{\bar x}\subseteq
\gcl{B_{0}}$; otherwise $f_{a,\bar b}(\bar x)=x_{0}\in\gcl{B_{0}}$. Thus
$\gcl{B_{0}}$ is closed under $\mathfrak F$, so
$\langle B_{0}\rangle_{\mathfrak F}\subseteq\gcl{B_{0}}$. Conversely, every
$a\in\gcl{B_{0}}$ lies in $\gcl{\bar b}$ for some tuple $\bar b$ of distinct
elements of $B_{0}$, and $f_{a,\bar b}(\bar b)=a$.

\emph{Condition (E3)} holds by the interpretation of $I$.

\emph{Condition (E2).} Each $s(\beta)$ is an automorphism of $M^\star$: it
preserves $I$ by Lemma~\ref{lem:coherent}(i), and for a tuple
$\bar x=\gamma\bar b$ of distinct elements of $B$,
\[
s(\beta)\bigl(f_{a,\bar b}(\bar x)\bigr)=s(\beta)s(\gamma)(a)=s(\beta\gamma)(a)
=f_{a,\bar b}(\beta\gamma\bar b)=f_{a,\bar b}\bigl(s(\beta)\bar x\bigr),
\]
while for the remaining $\bar x$ both sides equal $s(\beta)(x_{0})$. So every
permutation $\beta$ of $B$ extends to $s(\beta)\in\Aut(M^\star)$, and this
extension is unique in $\Aut(M^\star)$ because
$M=\gcl B=\langle B\rangle_{\mathfrak F}$ by Condition~(E1).
\end{proof}

\section{Proof of Lemma \ref{pro:main}}

\label{SECLEM}

Most parts of Lemma~\ref{pro:main} are straightforward. The substantial one is
the surjectivity in (4), which follows from the next lemma; its proof is close
to that of Lemma 3.2.19 in \cite{Zthesis}.

Throughout this section $A\subseteq_{\mathrm{fin}}M$ is fixed, $X:=\gcl A$, and we may
assume $A\leqslant M$, replacing $A$ by $\cl A$ if necessary (this does not
change $X$). We say that an infinite $Y\subseteq M$ is $\leqslant$\emph{-closed
in} $M$, written $Y\leqslant M$, if $Y$ is the union of a chain of finite
$\leqslant$-closed subsets of $M$; Lemma~\ref{lem:basic-generic-delta}(2),(3)
extend to this case.

\begin{claimnum}\label{cl:XinM}
\begin{enumerate}
\item For finite $Y\subseteq X$: $Y\leqslant X$ if and only if $Y\leqslant M$.
\item Every $Y$ with $A\subseteq Y\leqslant M$ and $Y\subseteq X$ satisfies
$\delta(Y)=\delta(A)$.
\item $X\leqslant M$, and $g[Y]\leqslant M$ for every $g\in\Aut(X)$ and every
finite $Y\leqslant M$ with $Y\subseteq X$.
\end{enumerate}
\end{claimnum}

\begin{proof}
(1) Only ``$\Rightarrow$'' needs an argument. Since $\dim{\cl Y}=\dim Y$ we
have $\cl Y\subseteq\gcl Y\subseteq\gcl X=X$, so $Y\subseteq\cl Y\subseteq X$
and $Y\leqslant X$ gives $\delta(\cl Y)\geq\delta(Y)$; as
$\delta(\cl Y)\leq\delta(Y)$ always, $\cl Y=Y$.
(2) $\delta(Y)=\dim Y=\dim A=\delta(A)$, the middle equality because
$A\subseteq Y\subseteq\gcl A$.
(3) $X=\bigcup\{\cl F: F\subseteq_{\mathrm{fin}}X\}$ and each $\cl F\subseteq X$; the
second assertion follows from (1), since $g$ preserves $\leqslant$-closedness
in $X$.
\end{proof}

\begin{lem}\label{lem:extend}
Let $X=\gcl A$ with $A\subseteq_{\mathrm{fin}}M$ and let $g\in\Aut(X)$. Then some
$\gamma\in\Aut(M)$ extends $g$.
\end{lem}

\begin{proof}
Fix chains $\langle B_{i}:i<\omega\rangle$ and $\langle C_{i}:i<\omega\rangle$
of finite $\leqslant$-closed subsets of $M$ with $B_{0}:=A$,
$C_{0}:=g[B_{0}]$ and $\bigcup_{i}B_{i}=\bigcup_{i}C_{i}=M$. We build partial
isomorphisms $g_{0}\leqslant g_{1}\leqslant\cdots$ between finite
$\leqslant$-closed subsets of $M$, starting from $g_{0}:=g_{\upharpoonright A}$,
maintaining for every $i$:
\begin{enumerate}
\item[(a)] $\dom{g_{i}}\leqslant M$ and $\ran{g_{i}}\leqslant M$;
\item[(b)] ${g_{i}}_{\upharpoonright(\dom{g_{i}}\cap X)}
=g_{\upharpoonright(\dom{g_{i}}\cap X)}$ and
${g_{i}}^{-1}_{\upharpoonright(\ran{g_{i}}\cap X)}
=g^{-1}_{\upharpoonright(\ran{g_{i}}\cap X)}$;
\item[(c)] $g_{i}[\dom{g_{i}}\cap X]=\ran{g_{i}}\cap X$. 
\end{enumerate}
At stage $i=2k$ we arrange $B_{k}\subseteq\dom{g_{i+1}}$ and at stage
$i=2k+1$ we arrange $C_{k}\subseteq\ran{g_{i+1}}$; since
$\bigcup_{i}B_{i}=\bigcup_{i}C_{i}=M$, the union
$\gamma:=\bigcup_{i<\omega}g_{i}$ is then an automorphism of $M$, and
$\gamma\upharpoonright X=g$ by (b). 

We describe the forth step; the back step is symmetric. Let
$D:=\dom{g_{i}}$ and suppose $B_{k}\not\subseteq D$ (otherwise take
$g_{i+1}:=g_{i}$). Put $\hat B:=\cl{B_{k}D}$, a finite $\leqslant$-closed
subset of $M$.

First, $DX\leqslant M$. Indeed, let $X'$ be finite with
$A\subseteq X'\leqslant X$; then $X'\leqslant M$ and $D\cap X'\leqslant X'$ by
Claim~\ref{cl:XinM} and Lemma~\ref{lem:basic-generic-delta}(2), so
$\delta(X')=\delta(D\cap X')=\delta(A)$ by Claim~\ref{cl:XinM}(2), and hence
\[
\dim{D}\leq\dim{DX'}\leq\delta(DX')\leq\delta(D)+\delta(X')-\delta(D\cap X')
=\delta(D)=\dim D .
\]
Thus $\delta(DX')=\dim{DX'}$, i.e.\ $DX'\leqslant M$, for every such $X'$;
as these $X'$ are cofinal in $X$, $DX\leqslant M$. 

Since $D\subseteq\hat B$ we have $D\cup(X\cap\hat B)=(D\cup X)\cap\hat B$,
which is $\leqslant$-closed in $\hat B$ by
Lemma~\ref{lem:basic-generic-delta}(2). Write
$P:=D\cup(X\cap\hat B)$ and
$h:=g_{i}\cup g\upharpoonright(\hat B\cap X)$. By (b), $g_{i}$ and $g$ agree
on $D\cap X$, so $h$ is a well-defined map with domain $P$. It is an isomorphism: the computation above is an equality, so there are no edges between $D\setminus X$ and $X\setminus D$, and likewise between $\ran{g_{i}}\setminus X$ and $X\setminus\ran{g_{i}}$; its range is $g_{i}[D]\cup g[\hat B\cap X]$, which is $\leqslant$-closed in $M$
by Claim~\ref{cl:XinM}(3) and the computation above applied to $\ran{g_{i}}$.

Since $P\leqslant\hat B$ and $\hat B\in\mathcal K_{0}$, genericity
(Theorem~\ref{thm-generic}(3)), transported along $h$, yields a
$\leqslant$-closed $E\subseteq M$ and an isomorphism
$\hat h:\hat B\to E$ extending $h$. 
Moreover
\[
E\cap X=g[\hat B\cap X]. \tag{$\dagger$}
\]
Indeed, $g[A]\leqslant M$ by Claim~\ref{cl:XinM}(3) and
$\delta(g[A])=\delta(A)=\dim X$, so $X=\gcl{g[A]}$. If
$b\in\hat B\setminus P$ then $b\notin X$, so $\dim{b/A}>0$; as $\hat h$ is an
isomorphism between $\leqslant$-closed sets extending $g_{\upharpoonright A}$,
it maps $\cl{Ab}\subseteq\hat B$ onto $\cl{g[A]\hat h(b)}$, whence
$\dim{\hat h(b)/g[A]}=\dim{b/A}>0$ and $\hat h(b)\notin X$. Therefore
$E\cap X=h[P]\cap X=g[\hat B\cap X]$, using (b) and (c) for $g_{i}$ and
$D\subseteq\hat B$.
Set $g_{i+1}:=\hat h$. Conditions (a) and (b) are immediate, and (c) holds
by~$(\dagger)$.
\end{proof}

\begin{proof}[Proof of Lemma~\ref{pro:main}]
(1) Let $A':=\cl A$. Every $\sigma\in G_{A'}$ satisfies
$\sigma[X]=\sigma[\gcl{A'}]=\gcl{A'}=X$, so $G_{A'}\leq G_{\{X\}}$ and
$G_{\{X\}}$ is open; an open subgroup of a topological group is closed, so
$G_{\{X\}}$ is closed, hence Polish. 

(2) $[G:G_{A'}]$ equals the number of images of $A'$ under $G$, which is
countable; so $[G:G_{\{X\}}]\leq\aleph_{0}<2^{\aleph_{0}}$.

(3) (A special case of Theorem 5.1.5 in \cite{Zthesis}.) Let $H\leq G$ with
$[G:H]<2^{\aleph_{0}}$. Then $H':=H\cap G_{\{X\}}$ satisfies
$[G_{\{X\}}:H']\leq[G:H]<2^{\aleph_{0}}$. If $G_{\{X\}}$ has the SIP then $H'$
is open in $G_{\{X\}}$, hence open in $G$ by (1); and $H\supseteq H'$ is a
union of cosets of $H'$, so $H$ is open in $G$. 

(4) If $\sigma\in G_{\{X\}}$ then $\sigma_{\upharpoonright X}\in\Aut(X)$, so
$\rho_{X}$ is a well-defined group homomorphism, and
$\ker\rho_{X}=G_{X}$. Surjectivity follows from Lemma~\ref{lem:extend}: given
$g\in\Aut(X)$, any $\gamma\in\Aut(M)$ extending $g$ lies in $G_{\{X\}}$.
For continuity and openness it suffices to consider a basic identity
neighbourhood $\Aut_{X_{0}}(X)$ of $\Aut(X)$, where $X_{0}\subseteq_{\mathrm{fin}}X$:
we have $\rho_{X}^{-1}(\Aut_{X_{0}}(X))=G_{X_{0}}\cap G_{\{X\}}$, a basic
identity neighbourhood of $G_{\{X\}}$, and
$\rho_{X}(G_{X_{0}}\cap G_{\{X\}})=\Aut_{X_{0}}(X)$ by surjectivity. A
homomorphism carrying a neighbourhood base of the identity onto open sets is
open. 

\end{proof}

\section{The small index property of $\protect\Aut(\hat M)$}
\label{SEC:4}

In this section we prove Proposition~C. The key case is $A=\emptyset$, where
$X=\hat M$: we show that $\hat M$ is the $(\mathcal C,\leqslant)$-generic
structure (Proposition~\ref{prop:Mhat-generic}), that $(\mathcal C,\leqslant)$
has the extension property (Lemma~\ref{lem:ep-im}), and hence that $\mathcal C$
is an amalgamation base and $\Aut(\hat M)$ has ample generics
(Theorem~\ref{thm:Mhat-SIP}). The general case is obtained at the end by
relativising over $\cl A$.

\noindent\textbf{Conventions for this section.} $\delta$, $\mathcal K_{0}$,
$M$ and $G=\Aut(M)$ are as in Section~\ref{subsection-setting}. Recall that
$\mathcal C=\{A\in\mathcal K_{0}:\delta(A)=0\}$ and
$\hat M=\gcl\emptyset$. 
The combinatorial content
is clause~(a) of  Definition~\ref{def:base0}(3), which in the present setting reads as
follows.

\begin{defn}\label{def-ext}
Let $\mathcal E\subseteq\mathcal K_{0}$. We say that $\mathcal E$ has the
\emph{extension property} (EP) if for every $A\in\mathcal E$ and all partial
$\leqslant$-isomorphisms $e_{1},\dots,e_{n}$ between $\leqslant$-closed
subsets of $A$ there are $B\in\mathcal E$ with $A\subseteq B$ and
$f_{i}\in\Aut(B)$ with $e_{i}\leqslant f_{i}$ for $1\leq i\leq n$. 
\end{defn}

\begin{rem}\label{rem:eppa}
Definition~\ref{def-ext} differs from the extension property for partial
automorphisms (EPPA) of the combinatorial literature in the following way: the partial
maps are required to be between $\leqslant$-closed subsets rather than
arbitrary ones. 
See Remark~\ref{rem:EP-literature} for the
relation to recent work on EPPA.
\end{rem}

\begin{rem}\label{rem-EP}
It is shown in \cite[Cor.~5.1.15]{Zthesis} that
$(\mathcal K_{0},\leqslant)$ does not have EP; in fact EP already fails for
some $A\in\mathcal K_{0}$ with a single partial $\leqslant$-isomorphism.
Similarly one shows that $(\mathcal K_{0}^{A},\leqslant)$ fails EP, where
$A\subseteq_{\mathrm{fin}}M$ with $\dim A>0$ and
$\mathcal K_{0}^{A}:=\{B\in\mathcal K_{0}:A\leqslant B\}$. 
For classes obtained from pre-dimensions with irrational coefficients, and for
the simple $\omega$-categorical generic structures of
\cite{MR1900903,MR1982948}, EP also fails, by a variant of the argument in
which at least two partial $\leqslant$-isomorphisms must be considered. A
connection between the failure of EP and the existence of a \emph{tree-pair}
was observed in \cite{GKP-Amen}; D.~M.~Evans (personal communication) has
given a different proof of the failure of EP for the classes obtained from
pre-dimensions with rational or irrational coefficients.
\end{rem}

The technical result of this section is the following; its proof is given at
the end.

\begin{lem}\label{lem:ep-im}
Let $A_{0}\in\mathcal K_{0}$, and let $\mathcal C_{A_{0}}$ be the class of all
$B\in\mathcal K_{0}$ with $A_{0}\leqslant B$ and $\delta(B)=\delta(A_{0})$. 
Let $A\in\mathcal C_{A_{0}}$, and let $e_{1},\dots,e_{n}$ be isomorphisms between $\leqslant$-closed subsets of $A$ containing $A_{0}$, each fixing $A_{0}$ point-wise. 
Then there are $B\in\mathcal C_{A_{0}}$ with $A\leqslant B$ and automorphisms $f_{1},\dots,f_{n}$ of $B$, fixing $A_{0}$ point-wise, with
$e_{i}\leqslant f_{i}$. In particular, taking $A_{0}=\emptyset$, the class
$ (\mathcal C,\leqslant)$ has the extension property.
\end{lem}

\begin{prop}\label{prop:Mhat-generic}
$(\mathcal C,\leqslant)$ has the free amalgamation
property, and $\hat M=\gcl\emptyset$ is the $(\mathcal C,\leqslant)$-generic
structure. Moreover, for $A,B\in\mathcal C$ with $A\subseteq B$ one has
$A\leqslant B$.
\end{prop}

\begin{proof}

If $A,B\in\mathcal C$ with $A\subseteq B$ then every $A'$ with
$A\subseteq A'\subseteq B$ satisfies $\delta(A')\geq 0=\delta(A)$, so
$A\leqslant B$. Free amalgamation: if $A\leqslant B,C$ with
$A,B,C\in\mathcal C$ then $B\otimes_{A}C\in\mathcal K_{0}$ by
Fact~\ref{free-amalgam} and $\delta(B\otimes_{A}C)=0$.

Finally, $\hat M=\{a\in M:\dim a=0\}$, so $A\subseteq\hat M$ is finite and
$\leqslant$-closed if and only if $A\in\mathcal C$; thus the finite
$\leqslant$-closed subsets of $\hat M$ are exactly the members of
$\mathcal C$, $\hat M$ is the union of a chain of them, and the extension
property of Theorem~\ref{thm-generic}(3) restricts to
$(\mathcal C,\leqslant)$.
\end{proof}

\begin{cor}\label{cor:amal}
$\mathcal C$, viewed as the family of finite $\leqslant$-closed substructures
of $\hat M$, is an amalgamation base for $\hat M$ in the sense of
Definition~\ref{def:base0}(3).
\end{cor}

\begin{proof}
$\mathcal C$ is countable and $\Aut(\hat M)$-invariant, and the stabilisers
$\Aut_A(\hat M)$, $A\in\mathcal C$, are open. It is cofinal: if
$X\subseteq_{\mathrm{fin}}\hat M$ then $\dim X=0$, so $\delta(\cl X)=0$ and
$\cl X\in\mathcal C$. Hence $\mathcal C$ is a base. %

For clause (a), the maps occurring in Definition~\ref{def:base0}(3)(a) are
restrictions of automorphisms of $\hat M$ to members of $\mathcal C$ contained
in $A$, that is, partial $\leqslant$-isomorphisms between $\leqslant$-closed
subsets of $A$. Lemma~\ref{lem:ep-im} (with $A_{0}=\emptyset$) gives
$B\in\mathcal C$ with $A\leqslant B$ and automorphisms $f_{i}$ of $B$ extending
them. By Proposition~\ref{prop:Mhat-generic} and
Theorem~\ref{thm-generic}(3), $B$ may be embedded over $A$ into $\hat M$, so
we may assume $B\subseteq\hat M$.

For clause (b), let $A,B,C\in\mathcal C$ with $A\subseteq B,C$; then
$A\leqslant B$ and $A\leqslant C$. By Fact~\ref{free-amalgam} and
Theorem~\ref{thm-generic}(3) there is a copy $B'$ of $B$ over $A$ inside $M$
such that $B'\cap C=A$ and $B'C$ is the free amalgam $B'\otimes_{A}C$, with $B'C\leqslant M$; since $\delta(B')=0$ we have
$\dim {B'}=0$, so $B'\subseteq\hat M$ and $B'\in\mathcal C$. 
By $\leqslant$-homogeneity (Theorem~\ref{thm-generic}(4)) there is
$\sigma\in G_{A}$ with $\sigma[B]=B'$; every automorphism of $M$ preserves
$\gcl\emptyset$ set-wise, so $\alpha:=\sigma_{\upharpoonright\hat M}$ lies in
$\Aut(\hat M)_{A}$ and $\alpha[B]=B'$. 
Finally, if $g\in\Aut(B')$ and $h\in\Aut(C)$ agree on $A$, then $g\cup h$ is a
well-defined isomorphism (as $B'\cap C=A$ and there are no edges between
$B'\setminus A$ and $C\setminus A$) between $\leqslant$-closed subsets
of $\hat M$, and it extends to an automorphism of $\hat M$ by
$\leqslant$-homogeneity of $\hat M$
(Proposition~\ref{prop:Mhat-generic}). 
\end{proof}

\begin{thm}\label{thm:Mhat-SIP}
$\Aut(\hat M)$ has the small index property.
\end{thm}

\begin{proof}
By Corollary~\ref{cor:amal}, $\mathcal C$ is an amalgamation base for $\hat M$,
with $\Gamma=\Aut(\hat M)$ and $\tau=\tau_{\mathrm{pt}}$. By
Proposition~\ref{prop:Gdelta}, for every $n\geq1$ the $\mathcal C$-generic
$n$-tuples form a dense $G_{\delta}$ subset of $\Aut(\hat M)^{n}$;    by Proposition~\ref{prop:conjugate}, applied with $B=\emptyset\in\mathcal C$, any
   two of them are conjugate. So $\Aut(\hat M)$ has ample generics, and hence the
   SIP by \cite[Thm.~1.6]{MR2308230}.
\end{proof}

\subsection*{Zero-sets and components}

Throughout this subsection and the next, fix $A_{0}\in\mathcal K_{0}$, and recall from Lemma~\ref{lem:ep-im} the class  $\mathcal C_{A_{0}}$ of all
$B\in\mathcal K_{0}$ with $A_{0}\leqslant B$ and $\delta(B)=\delta(A_{0})$; thus
$\mathcal C_{\emptyset}=\mathcal C$. For $B\in\mathcal C_{A_{0}}$, a
\emph{zero-set} of $B$ is a set $X$ with $A_{0}\subseteq X\subseteq B$ and
$\delta(X)=\delta(A_{0})$. The following is an easy consequence of
Lemma~\ref{lem:basic-generic-delta}(1) and of $A_{0}\leqslant B$.

\begin{rem}\label{lem:zero-sets}
	Let $B\in\mathcal C_{A_{0}}$.
	\begin{enumerate}
		\item The zero-sets of $B$ are exactly the $\leqslant$-closed subsets of $B$
		containing $A_{0}$.
		\item If $X,Y$ are zero-sets of $B$, then so are $X\cup Y$ and $X\cap Y$, and
		$R(X\setminus Y;Y\setminus X)=\emptyset$.
	\end{enumerate}
\end{rem}

By Remark~\ref{lem:zero-sets}(2), every $v\in B\setminus A_{0}$ lies in a
smallest zero-set $Z(v)$, and for $u,v\in B\setminus A_{0}$ we have
$u\in Z(v)$ if and only if $Z(u)\subseteq Z(v)$. We write $u\sim v$ if
$Z(u)=Z(v)$ and call the $\sim$-classes the \emph{components} of $B$. For a
component $K$ we put $Z(K):=Z(v)$ for any $v\in K$, and call
$S_{K}:=Z(K)\setminus K$ the \emph{base} of $K$. The \emph{level} of $K$ is
defined by induction on $|Z(K)|$: it is $0$ if $S_{K}=A_{0}$, and otherwise
$1+\max\{\mathrm{level}(K'):K'\subseteq S_{K}\text{ a component}\}$. For
$q\geq-1$ let $B_{\leq q}$ be the union of $A_{0}$ and all components of level
at most $q$; so $B_{\leq-1}=A_{0}$.

\begin{lem}\label{lem:components}
	Let $B\in\mathcal C_{A_{0}}$ and let $K$ be a component of $B$.
	\begin{enumerate}
		\item $S_{K}$ is a zero-set, and it is the union of $A_{0}$ with components of
		level smaller than that of $K$. Each $B_{\leq q}$ is a zero-set.
		\item If $u\in K$ and $w\in B\setminus K$ are joined by an edge, then either
		$w\in S_{K}$, or $w$ lies in a component $K'$ with $K\subseteq S_{K'}$, so that
		$\mathrm{level}(K')>\mathrm{level}(K)$. In particular there are no edges
		between distinct components of the same level.
		\item Let $e:D\to R$ be an isomorphism between zero-sets of $B$ fixing $A_{0}$
		point-wise. Then every component of $B$ is either contained in $D$ or disjoint
		from it, and $e$ maps each component $K\subseteq D$ onto a component $e[K]$ of
		the same level, with $e[S_{K}]=S_{e[K]}$.
	\end{enumerate}
\end{lem}

\begin{proof}
	(1) Let $u\in S_{K}\setminus A_{0}$. Then $Z(u)\subseteq Z(K)$, and
	$Z(u)\cap K=\emptyset$: if $w\in K\cap Z(u)$, then $Z(K)=Z(w)\subseteq Z(u)$,
	so $u\sim w$ and $u\in K$. Hence
	$S_{K}=A_{0}\cup\bigcup\{Z(u):u\in S_{K}\setminus A_{0}\}$ is a zero-set by
	Remark~\ref{lem:zero-sets}(2). Every component $K'\subseteq S_{K}$ satisfies
	$Z(K')\subsetneq Z(K)$ and so has smaller level. Finally
	$B_{\leq q}=A_{0}\cup\bigcup\{Z(K):\mathrm{level}(K)\leq q\}$ is a union of
	zero-sets.
	
	(2) If $w\in A_{0}$ then $w\in S_{K}$. Otherwise, by
	Remark~\ref{lem:zero-sets}(2) applied to $Z(u)$ and $Z(w)$, the edge does not
	run between $Z(u)\setminus Z(w)$ and $Z(w)\setminus Z(u)$. So either
	$w\in Z(u)$, whence $w\in S_{K}$; or $u\in Z(w)$ and $Z(u)\subsetneq Z(w)$, in
	which case $K\subseteq S_{K'}$ for the component $K'$ of $w$.
	
	(3) The zero-sets of $B$ contained in $D$ are exactly the sets $X$ with
	$A_{0}\subseteq X\subseteq D$ and $\delta(X)=\delta(A_{0})$. Since $e$ fixes $A_{0}$ and
	preserves $\delta$, it maps them bijectively and inclusion-preservingly onto
	the zero-sets of $B$ contained in $R$. As $Z(v)\subseteq D$ for
	$v\in D\setminus A_{0}$, it follows that $e[Z(v)]=Z(e(v))$. Hence $e$ preserves
	$\sim$ and the relation $u\in Z(v)$, which gives the assertions about
	components, bases and, by induction, levels. If a component $K$ meets $D$ in
	$v$, then $K\subseteq Z(v)\subseteq D$.
\end{proof}

\begin{proof}[Proof of Lemma~\ref{lem:ep-im}]
	Let $A\in\mathcal C_{A_{0}}$ and let $e_{1},\dots,e_{n}$ be as in the
	statement, with $D_{i}:=\dom{e_{i}}$. By Remark~\ref{lem:zero-sets}(1),
	$D_{i}$ and $\ran{e_{i}}$ are zero-sets of $A$. If $A=A_{0}$ there is nothing
	to prove; otherwise let $h$ be the largest level of a component of $A$, and for
	$0\leq q\leq h$ let $\mathfrak K_{q}$ be the set of components of $A$ of level
	$q$.
	
	Let $T$ be a zero-set of some structure in $\mathcal C_{A_{0}}$. A set $K'$
	disjoint from $T$ is a \emph{copy over $T$} of $K\in\mathfrak K_{q}$ if there is
	an isomorphism $S_{K}K\to TK'$ which maps $S_{K}$ onto $T$ and fixes $A_{0}$
	point-wise. 
Let $P_{q}(T)$ be the set of all types $\mathrm{tp}(K'/T)$, where $K'$ is a
copy over $T$ of some $K\in\mathfrak K_{q}$; it is finite since $\mathfrak K_{q}$
is finite.
	 An isomorphism $g:T\to T'$ fixing $A_{0}$
	point-wise induces a bijection 
	from $P_{q}(T)$ onto
	$P_{q}(T')$. For $p\in P_{q}(T)$ let $N(p)$ be the maximum, over all zero-sets
	$T'\subseteq A_{\leq q-1}$ of $A$ and all isomorphisms $g:T\to T'$ fixing
	$A_{0}$ point-wise, of the number of $K\in\mathfrak K_{q}$ with $S_{K}=T'$ and
	$\mathrm{tp}(K/T')=g(p)$. Clearly $N(g(p))=N(p)$ for every such $g$.
	
	We construct finite structures
	$A_{0}=B_{-1}\subseteq B_{0}\subseteq\cdots\subseteq B_{h}$ in
	$\mathcal C_{A_{0}}$, together with automorphisms $f_{i,q}$ of $B_{q}$ fixing
	$A_{0}$ point-wise, such that $A_{\leq q}$ is an induced substructure of
	$B_{q}$ and $f_{i,q}$ extends ${e_{i}}_{\upharpoonright D_{i}\cap A_{\leq q}}$.
	For $q=-1$ we take $f_{i,-1}:=\mathrm{id}_{A_{0}}$ note that
	$e_{i}$ fixes $A_{0}$ point-wise.
	
	Suppose that $B_{q}$ and the $f_{i,q}$ have been constructed, and $q<h$. We
	obtain $B_{q+1}$ by adding, for every zero-set $T$ of $B_{q}$ and every
	$p\in P_{q+1}(T)$, a set $\mathrm{Cop}(T,p)$ of $N(p)$ pairwise disjoint new
	realisations of $p$ over $T$, with no edges other than those inside the sets
	$TK'$ for $K'\in\mathrm{Cop}(T,p)$. When $T\subseteq A_{\leq q}$, the
	components $K\in\mathcal K_{q+1}$ with $S_{K}=T$ and $\mathrm{tp}(K/T)=p$ are
	taken to be among the members of $\mathrm{Cop}(T,p)$; this is possible because
	there are at most $N(p)$ of them. In this way each $K\in\mathcal K_{q+1}$ is
	used exactly once, namely for $T=S_{K}$, which is a zero-set of $A_{\leq q}$ by
	Lemma~\ref{lem:components}(1) and therefore a zero-set of $B_{q}$.
	
	The structure $B_{q+1}$ lies in $\mathcal C_{A_{0}}$. Indeed, each zero-set $T$
	of $B_{q}$ is $\leqslant$-closed in $B_{q}$ by Remark~\ref{lem:zero-sets}(1),
	and each $TK'$ is isomorphic to some $S_{K}K$, so $T\leqslant TK'$ and
	$\delta(TK')=\delta(T)$. Thus $B_{q+1}$ arises from $B_{q}$ by finitely many
	free amalgamations over $\leqslant$-closed sets, none of which changes
	$\delta$, and Fact~\ref{free-amalgam} together with
	Lemma~\ref{lem:basic-generic-delta}(3) gives $B_{q+1}\in\mathcal K_{0}$,
	$A_{0}\leqslant B_{q}\leqslant B_{q+1}$ and $\delta(B_{q+1})=\delta(A_{0})$. Moreover
	$A_{\leq q+1}$ is an induced substructure of $B_{q+1}$: by
	Lemma~\ref{lem:components}(2), every edge of $A_{\leq q+1}$ which is not in
	$A_{\leq q}$ lies inside some $K\in\mathcal K_{q+1}$ or runs between $K$ and
	$S_{K}$, and the same is true in $B_{q+1}$ by construction.
	
	It remains to extend the automorphisms. Fix $i$ and put $f:=f_{i,q}$. Then $f$
	permutes the zero-sets of $B_{q}$, and for each of them $p\mapsto f(p)$ is a
	bijection from $P_{q+1}(T)$ onto $P_{q+1}(f[T])$ with $N(f(p))=N(p)$; in
	particular $\mathrm{Cop}(T,p)$ and $\mathrm{Cop}(f[T],f(p))$ have the same
	size. We choose a bijection
	$\pi:\mathrm{Cop}(T,p)\to\mathrm{Cop}(f[T],f(p))$ and, for each
	$K'\in\mathrm{Cop}(T,p)$, an isomorphism $\psi_{K'}:TK'\to f[T]\pi(K')$
	extending $f_{\upharpoonright T}$. If $K'=K$ is a component of $A$ contained in
	$D_{i}$, then $T=S_{K}\subseteq D_{i}$ and
	$f_{\upharpoonright T}={e_{i}}_{\upharpoonright T}$; by
	Lemma~\ref{lem:components}(3), $e_{i}[K]$ is a component of level $q+1$ with
	base $f[T]$ and $\mathrm{tp}(e_{i}[K]/f[T])=f(p)$, so we may put
	$\pi(K):=e_{i}[K]$ and $\psi_{K}:={e_{i}}_{\upharpoonright TK}$. This
	prescribes $\pi$ injectively on part of $\mathrm{Cop}(T,p)$; we extend it to a
	bijection, and for the remaining $K'$ we take for $\psi_{K'}$ any isomorphism
	extending $f_{\upharpoonright T}$, which exists because
	$\mathrm{tp}(\pi(K')/f[T])=f(p)$.
	
	Let $f_{i,q+1}$ be the union of $f$ and all the maps $\psi_{K'}$. It is a
	bijection of $B_{q+1}$ fixing $A_{0}$ point-wise, and it preserves edges and
	non-edges, because every edge of $B_{q+1}$ not in $B_{q}$ lies in some $TK'$.
	Since every component of $A$ meeting $D_{i}$ is contained in $D_{i}$
	(Lemma~\ref{lem:components}(3)), $f_{i,q+1}$ extends
	${e_{i}}_{\upharpoonright D_{i}\cap A_{\leq q+1}}$.
	
	Finally put $B:=B_{h}$ and $f_{i}:=f_{i,h}$. Then $B\in\mathcal C_{A_{0}}$,
	$A=A_{\leq h}$ is an induced substructure of $B$, and each $f_{i}$ is an
	automorphism of $B$ fixing $A_{0}$ point-wise and extending $e_{i}$. Moreover
	$A\leqslant B$, since every $A'$ with $A\subseteq A'\subseteq B$ satisfies
	$\delta(A')\geq \delta(A_{0})=\delta(A)$.
\end{proof}

\subsection*{The relative case}

\begin{prop}\label{prop:CA}
Let $A\subseteq_{\mathrm{fin}}M$ with $A\leqslant M$ and $X=\gcl A$, and let
$\mathcal C_{A}(X)$ be the family of all $B\in\mathcal C_{A}$ with $B\subseteq X$.
Then $\mathcal C_{A}(X)$ is an amalgamation base for $X$ with respect to
$\Aut_{A}(X)$.
\end{prop}

\begin{proof}
By Claim~\ref{cl:XinM}, a finite $B$ with $A\subseteq B\subseteq X$ lies in
$\mathcal C_{A}(X)$ if and only if $B\leqslant X$, if and only if
$\delta(B)=\delta(A)$. Hence $\mathcal C_{A}(X)$ is countable, invariant under
$\Aut_{A}(X)$, and cofinal in $X$ (for $Y\subseteq_{\mathrm{fin}}X$ take
$\cl{AY}$); the stabilisers are open.

For clause (a), let $B_{1}\in\mathcal C_{A}(X)$ and let $e_{1},\dots,e_{n}$ be
restrictions of elements of $\Aut_{A}(X)$ to members of $\mathcal C_{A}(X)$
contained in $B_{1}$. These are isomorphisms between $\leqslant$-closed subsets
of $B_{1}$ containing $A$ which fix $A$ point-wise, so Lemma~\ref{lem:ep-im},
with $A_{0}:=A$, gives $B\in\mathcal C_A$ with $B_{1}\leqslant B$ and
automorphisms $f_{i}$ of $B$ extending the $e_{i}$ and fixing $A$ point-wise.
By Theorem~\ref{thm-generic}(3), $B$ embeds over $B_{1}$ into $M$ with
$\leqslant$-closed image; this image has pre-dimension $\delta(A)$, so it lies
in $\gcl A=X$ and belongs to $\mathcal C_{A}(X)$. Clause (b) as in Corollary~\ref{cor:amal}: the copy $B'$ of $B$ over $A$
obtained there satisfies $\delta(B')=\delta(A)$, so $B'\subseteq\gcl A=X$, and
the automorphisms of $M$ involved fix $A$ point-wise and hence preserve $X$.
\end{proof}

\begin{proof}[Proof of Proposition~\ref{thm:main-1}]
Let $A\subseteq_{\mathrm{fin}}M$ and $X=\gcl A$; we may assume $A\leqslant M$,
replacing $A$ by $\cl A$. Since $A$ is finite, $\Aut_{A}(X)$ is an open
subgroup of $\Aut(X)$, so by the argument of Lemma~\ref{pro:main}(3) it
suffices to show that $\Aut_{A}(X)$ has the SIP (cf.\
\cite[Thm.~5.1.5]{Zthesis}). 
By Proposition~\ref{prop:CA}, $\mathcal C_{A}(X)$ is an amalgamation base for $X$
with respect to $\Aut_{A}(X)$. Exactly as in the proof of
Theorem~\ref{thm:Mhat-SIP}, now applying Proposition~\ref{prop:conjugate}
with $B=A\in\mathcal C_{A}(X)$, the group $\Aut_{A}(X)$ has ample generics and
hence the SIP. 
\end{proof}

\section{Remaining cases}
\label{SEC:5}

Recall Hrushovski's collapse \cite{Hrunew}. Fix a function $\mu$ which
assigns a natural number to the isomorphism type of each pair $(Z,S)$ with
$Z$ $0$-minimally algebraic over $S$ (Definition~\ref{def:0-alg}), and let
$\mathcal K_{0}^{\mu}\subseteq\mathcal K_{0}$ consist of those
$A\in\mathcal K_{0}$ which contain at most $\mu(Z,S)$ pairwise disjoint copies
of $Z$ over any copy of $S$ in $A$. For a suitable choice of $\mu$ the class
$(\mathcal K_{0}^{\mu},\leqslant)$ has the amalgamation property; let
$M^{\mu}$ be its generic structure. The theory $\Th(M^{\mu})$ has finite
Morley rank, and in Hrushovski's original setting it is strongly minimal; this
is in contrast with the uncollapsed $M$ of this paper, whose theory has
infinite Morley rank.
 
\begin{rem}\label{rem:collapsed}
In $M^{\mu}$ the bound $\mu$ makes every element of geometric dimension $0$
over a finite set $A$ algebraic over $A$, so $\gcl A=\acl A$. If, moreover,
$M^{\mu}$ is almost strongly minimal (that is, $M^{\mu}\subseteq\acl{D(M^{\mu})}$
for some $\emptyset$-definable strongly minimal set $D$), then
Th\'eor\`eme~1 of \cite{Lascar} applies directly: every subgroup of small
index of $\Aut(M^{\mu})$ contains $\Aut_{\acleq{A}}(M^{\mu})$ for some finite
$A\subseteq M^{\mu}$. In the collapsed case, therefore, the analogue of
Theorem~A comes for free, and Section~\ref{SECTHM} is not needed.
 
The remaining step, however, is not easier than in the uncollapsed case. 
As with $\gcl A$ in $M$, the set $\acl A$ is not finite in general: every $0$-minimally algebraic extension of $\cl A$ which lies in $\mathcal K_{0}^{\mu}$ is realised in $M^{\mu}$, and there are infinitely many isomorphism types of these. 
On the other hand, every element of $\acl A$ has only finitely many conjugates over $A$, so the group of automorphisms of $\acl A$ induced by $\Aut_{A}(M^{\mu})$ is profinite.
The question is then whether its subgroups of small index are open, and this is not automatic: for the countable saturated algebraically closed field of characteristic $0$, Lascar's theorem holds with $\acl\emptyset=\bar{\mathbb Q}$, but $\mathrm{Gal}(\bar{\mathbb Q}/\mathbb Q)$ has non-open subgroups of finite index, and their preimages show that the SIP fails. Small index subgroups of $\Aut(M^{\mu})$ in the strongly minimal case are also discussed in \cite[Ch.~5]{Zthesis}.
\end{rem}
 
\begin{question}\label{q:collapsed}
Does $\Aut(M^{\mu})$ have the SIP when $M^{\mu}$ is one of Hrushovski's collapsed generic structures?
\end{question}

 \begin{rem}\label{rem:collapsed-EP}
The method of Section~\ref{SEC:4} does not carry over to the collapsed
setting. The proof of Lemma~\ref{lem:ep-im} adds, over \emph{every} zero-set of
$B_{q}$, the same number of copies of each component occurring in $A$, and in
general this exceeds the number of disjoint copies permitted by $\mu$; the
resulting structures will not belong to $\mathcal K_{0}^{\mu}$.
\end{rem}

\subsection*{Other pre-dimension functions}
 
\begin{question}\label{q:rational}
Does $\Aut(M)$ have the SIP when $M$ is the generic structure of the ab-initio class obtained from $\delta(A)=q|A|-p|R(A)|$ with $\gcd(p,q)=1$ and $1<p<q$?
\end{question}
 
As observed in Remark~\ref{rem:coefficients}, only the proof of Lemma~\ref{lem-sym} is sensitive to this change. In Case~3 of that proof, with $l:=\dim{m_{k}/A_{k}}$, one adds $r$ new vertices joined only to $m_{k}$, and the argument needs $rp=l$. 
For $p=1$ this is always possible, but for $p>1$ the value $l$ need not be divisible by $p$. Everything else in Sections~\ref{SECTHM}--\ref{SEC:4} goes through unchanged, so Question~\ref{q:rational} is really a question about Lemma~\ref{lem-sym}.
 
\begin{question}\label{q:irrational}
Does $\Aut(M)$ have the SIP when the coefficients of $\delta$ are irrational, or when $M$ is one of the simple $\omega$-categorical generic structures of
\cite{MR1900903,MR1982948}?
\end{question}
 
In these cases the extension property fails as well (Remark~\ref{rem-EP}), so Section~\ref{SEC:4} would have to be replaced. On the other hand, the $\omega$-categorical examples are covered by \cite[Thm.~5.3]{MR1231710} as soon as an amalgamation base is exhibited, so the obstruction is again a combinatorial one.
 
\begin{rem}\label{rem:EP-literature}
Both questions have a counterpart in the combinatorial literature on the extension property for partial automorphisms. By \cite[Thm.~3.4]{EHN19}, a finite graph lies in $\mathcal K_{0}$ if and only if it admits an orientation in which every out-degree is at most $\mathfrak m$. 
Thus $\mathcal K_{0}$ is the class of graph reducts of the class $\mathcal D^{\mathfrak m}$ of $\mathfrak m$-orientations, and $\mathcal D^{\mathfrak m}$, being a free amalgamation class in a language with relations and unary functions, has coherent EPPA \cite[Cor.~1.4]{HKN22}. Hubi\v{c}ka, Kone\v{c}n\'y and Ne\v{s}et\v{r}il moreover prove irreducible-structure-faithful coherent EPPA for the class $\mathcal D^{\mathfrak m}_{d}$ of $\mathfrak m$-orientations
with $d$-closures \cite[Thm.~12.9]{HKN22}, confirming \cite[Conj.~7.5]{EHN19}.
These results do not yield Lemma~\ref{lem:ep-im}, for two reasons.
First, they concern the expansion of $\mathcal K_{0}$ by an orientation, and a partial $\leqslant$-isomorphism of a graph in $\mathcal K_{0}$ need not preserve a chosen orientation. Second, an EPPA-witness produced there lies in $\mathcal K_{0}$ rather than in $\mathcal C$. 
It would nevertheless be interesting to know whether the methods of \cite{HKN22} can be adapted to give Lemma~\ref{lem:ep-im} directly; this would in addition yield \emph{coherent} extensions, and hence, by the argument of \cite{SinSol}, a dense locally
finite subgroup of $\Aut(\hat M)$.

\end{rem}
 
\subsection*{The strong small index property}
 
Recall from Section~\ref{SEC:intro} that $\Aut(M)$ has the \emph{strong} small index property (SSIP) if every subgroup of small index lies between $\Aut_{A}(M)$ and $\Aut_{\{A\}}(M)$ for some $A\subseteq_{\mathrm{fin}}M$. 
For an $\omega$-categorical structure $N$, the SSIP for $\Aut(N)$ implies that $\Th(N)$ admits weak elimination of imaginaries \cite[p.~161]{MR1221741}.
 
\begin{question}\label{q:ssip}
Does $\Aut(M)$ have the strong small index property?
\end{question}
 
We make two remarks on the present state of this question. 
First, the implication just quoted does not apply to $M$, which is not $\omega$-categorical, so weak elimination of imaginaries for $\Th(M)$, proved in \cite[Prop.~5.3]{Balshi}, is not by itself evidence for an affirmative
answer. 
For comparison, Baldwin and Verbovskiy \cite{BalVer} have shown, for Hrushovski's original (collapsed) strongly minimal construction and for the strongly minimal $k$-Steiner systems of Baldwin and Paolini, that the symmetric definable closure of an independent set is trivial, and hence that these theories do \emph{not} admit elimination of imaginaries; the weaker geometric elimination of imaginaries was already observed in
\cite[p.~160]{Hrunew}.
 
Second, a positive criterion is available. Paolini and Shelah \cite{PaoShe-SSIP} prove that a countable homogeneous structure which is algebraically locally finite and carries a \emph{free} stationary independence relation has the SSIP as soon as it has the SIP. 
The structure $M$ has locally finite algebraicity, by Fact~\ref{fact:cl}(1) and the finiteness of $\cl A$, and it carries a stationary independence relation derived from $\gcl-$ \cite[Lemma~4.3, Cor.~4.8]{ZEvansTent}. 
It is unclear whether the remaining hypotheses hold: $M$ is not ultrahomogeneous in the language $L=\{R\}$, being only $\leqslant$-homogeneous (Theorem~\ref{thm-generic}), and the freeness required in \cite{PaoShe-SSIP} would have to be checked for the independence relation coming from $\gcl-$ rather than from $\cl-$. 
Settling this would answer Question~\ref{q:ssip} for $M$.

\subsection*{Acknowledgement}
A major part of this paper was developed while the author was visiting the Mathematisches Institut of the Westf\"alische Wilhelms-Universit\"at M\"unster during the winter semester 2014/2015. 
The author is grateful to David M. Evans, who suggested Proposition~\ref{thm:main-1} in the current form, and to Katrin Tent, for many suggestions, comments and corrections on earlier versions of this paper; to John Baldwin for his comments on the exposition of the results and on the final section. 

\bibliographystyle{siam}
\bibliography{ALL}

@Article{EHN19,
  Author                   = {Evans, David M. and Hubi\v{c}ka, Jan and Ne\v{s}et\v{r}il, Jaroslav},
  Title                    = {Automorphism groups and {R}amsey properties of sparse graphs},
  Journal                  = {Proc. London Math. Soc. (3)},
  Fjournal                 = {Proceedings of the London Mathematical Society. Third Series},
  Year                     = {2019},
  Volume                   = {119},
  Number                   = {2},
  Pages                    = {515--546},
  ISSN                     = {0024-6115},
  Doi                      = {10.1112/plms.12238},
  Url                      = {http://dx.doi.org/10.1112/plms.12238}
}

@Article{HKN22,
  Author                   = {Hubi\v{c}ka, Jan and Kone\v{c}n\'y, Mat\v{e}j and Ne\v{s}et\v{r}il, Jaroslav},
  Title                    = {All those {EPPA} classes (strengthenings of the {H}erwig--{L}ascar theorem)},
  Journal                  = {Trans. Amer. Math. Soc.},
  Fjournal                 = {Transactions of the American Mathematical Society},
  Year                     = {2022},
  Volume                   = {375},
  Number                   = {11},
  Pages                    = {7601--7667},
  ISSN                     = {0002-9947},
  Doi                      = {10.1090/tran/8681},
  Url                      = {http://dx.doi.org/10.1090/tran/8681}
}

@Article{BalVer,
  Author                   = {Baldwin, John T. and Verbovskiy, Viktor V.},
  Title                    = {Towards a finer classification of strongly minimal sets},
  Journal                  = {Ann. Pure Appl. Logic},
  Fjournal                 = {Annals of Pure and Applied Logic},
  Year                     = {2024},
  Volume                   = {175},
  Number                   = {2},
  Pages                    = {103376},
  ISSN                     = {0168-0072},
  Doi                      = {10.1016/j.apal.2023.103376},
  Url                      = {http://dx.doi.org/10.1016/j.apal.2023.103376}
}

@Article{PaoShe-SSIP,
  Author                   = {Paolini, Gianluca and Shelah, Saharon},
  Title                    = {The strong small index property for free homogeneous structures},
  Note                     = {Preprint, arXiv:1703.10517},
  Year                     = {2018},
  Url                      = {http://arxiv.org/abs/1703.10517}
}

@Article{SinSol,
  Author                   = {Siniora, Daoud and Solecki, S{\l}awomir},
  Title                    = {Coherent extension of partial automorphisms, free amalgamation and automorphism groups},
  Journal                  = {J. Symbolic Logic},
  Fjournal                 = {The Journal of Symbolic Logic},
  Year                     = {2020},
  Volume                   = {85},
  Number                   = {1},
  Pages                    = {199--223},
  ISSN                     = {0022-4812},
  Doi                      = {110.1017/jsl.2019.32}
}

@article {MR859950,
    AUTHOR = {Dixon, John D. and Neumann, Peter M. and Thomas, Simon},
     TITLE = {Subgroups of small index in infinite symmetric groups},
   JOURNAL = {Bull. London Math. Soc.},
  FJOURNAL = {The Bulletin of the London Mathematical Society},
    VOLUME = {18},
      YEAR = {1986},
    NUMBER = {6},
     PAGES = {580--586},
      ISSN = {0024-6093},
     CODEN = {LMSBBT},
   MRCLASS = {20B35},
  MRNUMBER = {859950 (88i:20004)},
MRREVIEWER = {H. Dugald Macpherson},
       DOI = {10.1112/blms/18.6.580},
       URL = {http://dx.doi.org/10.1112/blms/18.6.580},
}

@Article{MR831437,
  Author                   = {Ahlbrandt, Gisela and Ziegler, Martin},
  Title                    = {Quasi-finitely axiomatizable totally categorical theories},
  Journal                  = {Ann. Pure Appl. Logic},
  Fjournal                 = {Annals of Pure and Applied Logic},
  Year                     = {1986},
  Volume                   = {30},
  Number                   = {1},
  Pages                    = {63--82},
  Note                     = {Stability in model theory (Trento, 1984)},
  ISSN                     = {0168-0072},
  Mrclass                  = {03C35 (03C45)},
  Mrnumber                 = {831437},
  Doi                      = {10.1016/0168-0072(86)90037-0},
  Url                      = {https://doi.org/10.1016/0168-0072(86)90037-0}
}

@InProceedings{LAS-AUT-SIP,
  Author                   = {Lascar, Daniel},
  Title                    = {Automorphism groups of saturated structures: a review},
  Booktitle                = {Proceedings of the {I}nternational {C}ongress of {M}athematicians, {V}ol. {II} ({B}eijing, 2002)},
  Publisher                = {Higher Ed. Press},
  Address                  = {Beijing},
  Year                     = {2002},
  Pages                    = {25--36},
  Url                      = {http://arxiv.org/abs/math/0304205}
}

@Article{GKP-Amen,
  Author                   = {Ghadernezhad, Zaniar and Khalilian, Hamed and Pourmahdian, Massoud},
  Title                    = {Automorphism groups of generic structures: extreme amenability and amenability},
  Journal                  = {Fund. Math.},
  Fjournal                 = {Fundamenta Mathematicae},
  Year                     = {2018},
  Volume                   = {242},
  Number                   = {1},
  Pages                    = {1--23},
  ISSN                     = {0016-2736},
  Doi                      = {10.4064/fm340-5-2017}
}

@Article{GhaVil,
  Author                   = {Ghadernezhad, Zaniar and Villaveces, Andr\'es},
  Title                    = {The small index property for homogeneous models in {AEC}'s},
  Journal                  = {Arch. Math. Logic},
  Fjournal                 = {Archive for Mathematical Logic},
  Year                     = {2018},
  Volume                   = {57},
  Number                   = {1-2},
  Pages                    = {141--157},
  ISSN                     = {0933-5846},
  Doi                      = {10.1007/s00153-017-0587-y},
  Url                      = {http://dx.doi.org/10.1007/s00153-017-0587-y}
}

@Article{PaoShe-Rec,
  Author                   = {Paolini, Gianluca and Shelah, Saharon},
  Title                    = {Reconstructing structures with the strong small index property up to bi-definability},
  Journal                  = {Fund. Math.},
  Fjournal                 = {Fundamenta Mathematicae},
  Year                     = {2019},
  Volume                   = {247},
  Number                   = {1},
  Pages                    = {25--35},
  ISSN                     = {0016-2736},
  Doi                      = {10.4064/fm640-11-2018}
}

@Article{HodOtt,
  Author                   = {Hodkinson, Ian and Otto, Martin},
  Title                    = {Finite conformal hypergraph covers and {G}aifman cliques in finite structures},
  Journal                  = {Bull. Symbolic Logic},
  Fjournal                 = {The Bulletin of Symbolic Logic},
  Year                     = {2003},
  Volume                   = {9},
  Number                   = {3},
  Pages                    = {387--405},
  ISSN                     = {1079-8986},
  Doi                      = {10.2178/bsl/1058448678}
}

@Article{ZEvansTent,
  AUTHOR = {Evans, David M. and Ghadernezhad, Zaniar and Tent, Katrin},
     TITLE = {Simplicity of the automorphism groups of some {H}rushovski
              constructions},
   JOURNAL = {Ann. Pure Appl. Logic},
  FJOURNAL = {Annals of Pure and Applied Logic},
    VOLUME = {167},
      YEAR = {2016},
    NUMBER = {1},
     PAGES = {22--48},
      ISSN = {0168-0072},
   MRCLASS = {03C15 (03C98 20B07 20B27)},
  MRNUMBER = {3413492},
MRREVIEWER = {O. V. Belegradek},
       DOI = {10.1016/j.apal.2015.09.002},
       URL = {http://dx.doi.org/10.1016/j.apal.2015.09.002},
}

@article {MR1438640,
    AUTHOR = {Bryant, Roger M. and Evans, David M.},
     TITLE = {The small index property for free groups and relatively free
              groups},
   JOURNAL = {J. London Math. Soc. (2)},
  FJOURNAL = {Journal of the London Mathematical Society. Second Series},
    VOLUME = {55},
      YEAR = {1997},
    NUMBER = {2},
     PAGES = {363--369},
      ISSN = {0024-6107},
     CODEN = {JLMSAK},
   MRCLASS = {20F28},
  MRNUMBER = {1438640},
MRREVIEWER = {J. R. J. Groves},
       DOI = {10.1112/S0024610796004711},
       URL = {http://dx.doi.org/10.1112/S0024610796004711},
}

@article {MR1368465,
    AUTHOR = {Kossak, Roman and Schmerl, James H.},
     TITLE = {Arithmetically saturated models of arithmetic},
      NOTE = {Special Issue: Models of arithmetic},
   JOURNAL = {Notre Dame J. Formal Logic},
  FJOURNAL = {Notre Dame Journal of Formal Logic},
    VOLUME = {36},
      YEAR = {1995},
    NUMBER = {4},
     PAGES = {531--546},
      ISSN = {0029-4527},
     CODEN = {NDJFAM},
   MRCLASS = {03C62 (03C57)},
  MRNUMBER = {1368465},
MRREVIEWER = {Fuxing Shen},
       DOI = {10.1305/ndjfl/1040136914},
       URL = {http://dx.doi.org/10.1305/ndjfl/1040136914},
}

@Book{Zthesis,
  Title                    = {{Automorphism groups of generic structures.}},
  Author                   = {Zaniar {Ghadernezhad}},
  Publisher                = {M\"unster: Univ. M\"unster, Mathematisch-Naturwissenschaftliche Fakult\"at, Fachbereich Mathematik und Informatik (Diss.)},
  Year                     = {2013},
  Language                 = {English},
  Msc2010                  = {03-02 03C60 08A35},
  Pages                    = {iv + 84}
}

@Article{MR1658539,
  Title                    = {Extending partial isomorphisms for the small index property of many {$\omega$}-categorical structures},
  Author                   = {Herwig, Bernhard},
  Journal                  = {Israel J. Math.},
  Year                     = {1998},
  Pages                    = {93--123},
  Volume                   = {107},
  Coden                    = {ISJMAP},
  Doi                      = {10.1007/BF02764005},
  Fjournal                 = {Israel Journal of Mathematics},
  ISSN                     = {0021-2172},
  Mrclass                  = {03C35 (03C65 05C99)},
  Mrnumber                 = {1658539 (2000c:03027)},
  Mrreviewer               = {Alessandro Berarducci},
  Url                      = {http://dx.doi.org/10.1007/BF02764005}
}

@Article{MR1357282,
  Title                    = {Extending partial isomorphisms on finite structures},
  Author                   = {Herwig, Bernhard},
  Journal                  = {Combinatorica},
  Year                     = {1995},
  Number                   = {3},
  Pages                    = {365--371},
  Volume                   = {15},
  Coden                    = {COMBDI},
  Doi                      = {10.1007/BF01299742},
  Fjournal                 = {Combinatorica. An International Journal on Combinatorics and the Theory of Computing},
  ISSN                     = {0209-9683},
  Mrclass                  = {03C13 (05C65 05C80 08A35)},
  Mrnumber                 = {1357282 (97a:03044)},
  Mrreviewer               = {H.-J. Hoehnke},
  Url                      = {http://dx.doi.org/10.1007/BF01299742}
}

@Article{MR1621745,
  Title                    = {Extending partial automorphisms and the profinite topology on free groups},
  Author                   = {Herwig, Bernhard and Lascar, Daniel},
  Journal                  = {Trans. Amer. Math. Soc.},
  Year                     = {2000},
  Number                   = {5},
  Pages                    = {1985--2021},
  Volume                   = {352},
  Coden                    = {TAMTAM},
  Doi                      = {10.1090/S0002-9947-99-02374-0},
  Fjournal                 = {Transactions of the American Mathematical Society},
  ISSN                     = {0002-9947},
  Mrclass                  = {20E05 (05C25 08A35 20E18)},
  Mrnumber                 = {1621745 (2000j:20036)},
  Mrreviewer               = {S. Andreadakis},
  Url                      = {http://dx.doi.org/10.1090/S0002-9947-99-02374-0}
}

@Article{MR1231710,
  Title                    = {The small index property for {$\omega$}-stable {$\omega$}-categorical structures and for the random graph},
  Author                   = {Hodges, Wilfrid and Hodkinson, Ian and Lascar, Daniel and Shelah, Saharon},
  Journal                  = {J. London Math. Soc. (2)},
  Year                     = {1993},
  Number                   = {2},
  Pages                    = {204--218},
  Volume                   = {48},

  Coden                    = {JLMSAK},
  Doi                      = {10.1112/jlms/s2-48.2.204},
  Fjournal                 = {Journal of the London Mathematical Society. Second Series},
  ISSN                     = {0024-6107},
  Mrclass                  = {03C45 (03C35 05C80 22A05)},
  Mrnumber                 = {1231710 (94d:03063)},
  Mrreviewer               = {Alexandre Ivanov},
  Url                      = {http://dx.doi.org/10.1112/jlms/s2-48.2.204}
}

@Article{Hrunew,
  Title                    = {A new strongly minimal set},
  Author                   = {Hrushovski, Ehud},
  Journal                  = {Ann. Pure Appl. Logic},
  Year                     = {1993},
  Note                     = {Stability in model theory, III (Trento, 1991)},
  Number                   = {2},
  Pages                    = {147--166},
  Volume                   = {62},
  Coden                    = {APALD7},
  Doi                      = {10.1016/0168-0072(93)90171-9},
  Fjournal                 = {Annals of Pure and Applied Logic},
  ISSN                     = {0168-0072},
  Mrclass                  = {03C45 (03C35)},
  Mrnumber                 = {1226304 (94d:03064)},
  Mrreviewer               = {Alexandre Ivanov},
  Url                      = {http://dx.doi.org/10.1016/0168-0072(93)90171-9}
}

@Article{HruEP,
  Title                    = {Extending partial isomorphisms of graphs},
  Author                   = {Hrushovski, Ehud},
  Journal                  = {Combinatorica},
  Year                     = {1992},
  Number                   = {4},
  Pages                    = {411--416},
  Volume                   = {12},
  Coden                    = {COMBDI},
  Doi                      = {10.1007/BF01305233},
  Fjournal                 = {Combinatorica. An International Journal on Combinatorics and the Theory of Computing},
  ISSN                     = {0209-9683},
  Mrclass                  = {05C25 (05C80)},
  Mrnumber                 = {1194731 (93m:05089)},
  Mrreviewer               = {Martin {\v{S}}koviera},
  Url                      = {http://dx.doi.org/10.1007/BF01305233}
}

@Article{MR2308230,
  Title                    = {Turbulence, amalgamation, and generic automorphisms of homogeneous structures},
  Author                   = {Kechris, Alexander S. and Rosendal, Christian},
  Journal                  = {Proc. Lond. Math. Soc. (3)},
  Year                     = {2007},
  Number                   = {2},
  Pages                    = {302--350},
  Volume                   = {94},
  Doi                      = {10.1112/plms/pdl007},
  Fjournal                 = {Proceedings of the London Mathematical Society. Third Series},
  ISSN                     = {0024-6115},
  Mrclass                  = {03E15 (37B05)},
  Mrnumber                 = {2308230 (2008a:03079)},
  Mrreviewer               = {Tam{{\'a}}s M{{\'a}}trai},
  Url                      = {http://dx.doi.org/10.1112/plms/pdl007}
}

@Article{Lascar,
  Title                    = {Les automorphismes d'un ensemble fortement minimal},
  Author                   = {Lascar, Daniel},
  Journal                  = {J. Symbolic Logic},
  Year                     = {1992},
  Number                   = {1},
  Pages                    = {238--251},
  Volume                   = {57},
  Coden                    = {JSYLA6},
  Doi                      = {10.2307/2275188},
  Fjournal                 = {The Journal of Symbolic Logic},
  ISSN                     = {0022-4812},
  Mrclass                  = {03C45 (03C15 20B27)},
  Mrnumber                 = {1150937 (93c:03042)},
  Mrreviewer               = {Anand Pillay},
  Url                      = {http://dx.doi.org/10.2307/2275188}
}

@Article{MR1204064,
  Title                    = {Uncountable saturated structures have the small index property},
  Author                   = {Lascar, Daniel and Shelah, Saharon},
  Journal                  = {Bull. London Math. Soc.},
  Year                     = {1993},
  Number                   = {2},
  Pages                    = {125--131},
  Volume                   = {25},

  Coden                    = {LMSBBT},
  Doi                      = {10.1112/blms/25.2.125},
  Fjournal                 = {The Bulletin of the London Mathematical Society},
  ISSN                     = {0024-6093},
  Mrclass                  = {03C50 (20B27)},
  Mrnumber                 = {1204064 (94d:03068)},
  Mrreviewer               = {Jacinta Covington},
  Url                      = {http://dx.doi.org/10.1112/blms/25.2.125}
}

@Article{MacS,
  Title                    = {A survey of homogeneous structures},
  Author                   = {Macpherson, Dugald},
  Journal                  = {Discrete Math.},
  Year                     = {2011},
  Number                   = {15},
  Pages                    = {1599--1634},
  Volume                   = {311},
  Coden                    = {DSMHA4},
  Doi                      = {10.1016/j.disc.2011.01.024},
  Fjournal                 = {Discrete Mathematics},
  ISSN                     = {0012-365X},
  Mrclass                  = {03C15 (03E15 05D10 20B27 22F50 68Q25)},
  Mrnumber                 = {2800979 (2012e:03063)},
  Mrreviewer               = {G. Cherlin},
  Url                      = {http://dx.doi.org/10.1016/j.disc.2011.01.024}
}

@InCollection{Wag1,
  Title                    = {Relational structures and dimensions},
  Author                   = {Wagner, Frank O.},
  Booktitle                = {Automorphisms of first-order structures},
  Publisher                = {Oxford Univ. Press},
  Year                     = {1994},

  Address                  = {New York},
  Pages                    = {153--180},
  Series                   = {Oxford Sci. Publ.},
  Mrclass                  = {03C45 (03C35)},
  Mrnumber                 = {1325473}
}

@Article{Balshi,
  Title                    = {Stable generic structures},
  Author                   = {Baldwin, John T. and Shi, Niandong},
  Journal                  = {Ann. Pure Appl. Logic},
  Year                     = {1996},
  Number                   = {1},
  Pages                    = {1--35},
  Volume                   = {79},
  Coden                    = {APALD7},
  Doi                      = {10.1016/0168-0072(95)00027-5},
  Fjournal                 = {Annals of Pure and Applied Logic},
  ISSN                     = {0168-0072},
  Mrclass                  = {03C45 (03C50)},
  Mrnumber                 = {1390325 (97c:03103)},
  Mrreviewer               = {O. V. Belegradek},
  Url                      = {http://dx.doi.org/10.1016/0168-0072(95)00027-5}
}

@Book{MR1425877,
  Title                    = {The descriptive set theory of {P}olish group actions},
  Author                   = {Becker, Howard and Kechris, Alexander S.},
  Publisher                = {Cambridge University Press},
  Year                     = {1996},

  Address                  = {Cambridge},
  Series                   = {London Mathematical Society Lecture Note Series},
  Volume                   = {232},
  Doi                      = {10.1017/CBO9780511735264},
  ISBN                     = {0-521-57605-9},
  Mrclass                  = {54H05 (04A15 04A20 28A05)},
  Mrnumber                 = {1425877 (98d:54068)},
  Mrreviewer               = {Klaas Pieter Hart},
  Pages                    = {xii+136},
  Url                      = {http://dx.doi.org/10.1017/CBO9780511735264}
}

@article {MR1900903,
    AUTHOR = {Evans, David M.},
     TITLE = {{$\aleph_0$}-categorical structures with a predimension},
   JOURNAL = {Ann. Pure Appl. Logic},
  FJOURNAL = {Annals of Pure and Applied Logic},
    VOLUME = {116},
      YEAR = {2002},
    NUMBER = {1-3},
     PAGES = {157--186},
      ISSN = {0168-0072},
     CODEN = {APALD7},
   MRCLASS = {03C45 (03C15 03C35)},
  MRNUMBER = {1900903 (2003m:03059)},
       DOI = {10.1016/S0168-0072(01)00109-9},
       URL = {http://dx.doi.org/10.1016/S0168-0072(01)00109-9},
}

@book {MR1221741,
    AUTHOR = {Hodges, Wilfrid},
     TITLE = {Model theory},
    SERIES = {Encyclopedia of Mathematics and its Applications},
    VOLUME = {42},
 PUBLISHER = {Cambridge University Press, Cambridge},
      YEAR = {1993},
     PAGES = {xiv+772},
      ISBN = {0-521-30442-3},
   MRCLASS = {03-01 (03-02 03Cxx)},
  MRNUMBER = {1221741},
MRREVIEWER = {J. M. Plotkin},
       DOI = {10.1017/CBO9780511551574},
       URL = {http://dx.doi.org/10.1017/CBO9780511551574},
}

@article {MR1982948,
    AUTHOR = {Pourmahdian, Massoud},
     TITLE = {Simple generic structures},
   JOURNAL = {Ann. Pure Appl. Logic},
  FJOURNAL = {Annals of Pure and Applied Logic},
    VOLUME = {121},
      YEAR = {2003},
    NUMBER = {2-3},
     PAGES = {227--260},
      ISSN = {0168-0072,1873-2461},
   MRCLASS = {03C45},
  MRNUMBER = {1982948},
MRREVIEWER = {Alexandre\ Ivanov},
       DOI = {10.1016/S0168-0072(02)00114-8},
       URL = {https://doi.org/10.1016/S0168-0072(02)00114-8},
}

\end{document}